\documentclass[10pt,a4paper]{article} 

\usepackage{anysize}
\usepackage[format = hang]{caption}
\usepackage{amsmath,amsthm,verbatim,amssymb,amsfonts,amscd, graphicx}
\usepackage{graphics,natbib}
\usepackage{titlesec,footmisc}
\usepackage{hyperref} 
\theoremstyle{plain}
\newtheorem{theorem}{Theorem}

\newtheorem{lemma}{Lemma}

\newtheorem{example}{Example}[section]
\theoremstyle{definition}

\newcommand{\cov}{\text{cov}}
\newcommand{\corr}{\text{corr}}

\def\sumip{\sum_{i=1}^p}
\def\E{{\rm E}}
\def\Var{{\rm Var}}
\def\tr{{\rm t}}
\def\midd{\,|\,}
\def\pr{{\rm pr}}
\def\dd{{\rm d}}
\def\hatt{\widehat}
\def\arr{\rightarrow}
\def\N{{\rm N}}
\def\half{\hbox{$1\over 2$}}
\def\eps{\varepsilon}
\def\Gam{{\rm Gamma}}

\def\Dir{{\rm Dir}}
\def\data{{\rm data}}

\def\beq{\begin{eqnarray}}
\def\eeq{\end{eqnarray}}

\def\beqn{\begin{eqnarray*}}  
\def\eeqn{\end{eqnarray*}}

\titleformat{\section}{\normalfont\large\sc\centering}{\thesection}{1em}{}
\titleformat{\subsection}[runin]{\normalfont\large\bfseries}{\thesubsection}{1em}{}
\numberwithin{equation}{section} 
\renewenvironment{abstract}
               {\list{}{\rightmargin\leftmargin}%
                \item[\text{\hspace{10mm}\sc Abstract.}]\relax}
               {\endlist}

\begin{document}

\def\today{April 2019}


\begingroup
\begin{centering} 
\large{\sc\bf Multivariate Estimation of Poisson Parameters}\\[0.8em]
\small {\sc Emil Aas Stoltenberg$^1$ 
   and Nils Lid Hjort$^2$} \\[0.3em] 
\small {\rm $^1$emilas@math.uio.no and $^2$nils@math.uio.no} \\[0.3em] 
\small {\sc Department of Mathematics, University of Oslo} \\[0.3em]
\small {\sc P.B.~1053, Blindern, N-0316 Oslo, Norway} \\[0.8em]
\small {\sc {\today}}\par
\end{centering}
\endgroup


\begin{abstract}
\small{This paper is devoted to the multivariate estimation 
of a vector of Poisson means. 
A novel loss function that penalises bad estimates of each 
of the parameters \textit{and} the sum (or equivalently the mean) 
of the parameters is introduced. Under this loss function, 
a class of minimax estimators that uniformly dominate 
the maximum likelihood estimator is derived. Crucially, 
these methods have the property that for estimating
a given component parameter, the full data vector is 
utilised. Estimators in this class can be fine-tuned to 
limit shrinkage away from the maximum likelihood estimator, 
thereby avoiding implausible estimates of the sum of the 
parameters. Further light is shed on this new class 
of estimators by showing that it can be derived by 
Bayesian and empirical Bayesian methods. In particular,
we exhibit a generalisation of the Clevenson--Zidek
estimator, and prove its admissibility. Moreover, 
a class of prior distributions for which the Bayes 
estimators uniformly dominate the maximum likelihood 
estimator under the new loss function is derived. 
A section is included involving weighted loss functions,
notably also leading to a procedure improving uniformly
on the maximum likelihood method in an infinite-dimensional
setup. Importantly, some of our methods lead to constructions
of new multivariate models for both rate parameters and 
count observations. Finally, estimators that shrink 
the usual estimators towards a data based point 
in the parameter space are derived and compared.

}
\end{abstract}



\section{Introduction}
\label{section:intro}

Let $Y = (Y_1,\ldots,Y_p)$ be a vector of 
independent Poisson random variables with mean vector 
$\theta = (\theta_1,\ldots,\theta_p)$. In this article 
we consider the problem of estimating the vector $\theta$. 
The obvious estimator is $\delta_0(Y) = Y$, that is, 
using $\delta_{0,i} = Y_i$ for each of the $p$ individual 
parameters. It is well known that $\delta_0$ is the 
maximum likelihood solution, that it has components 
with uniformly minimal variance among the unbiased 
estimators, and that it is admissible in the 
one-dimensional problem under squared error loss 
as well as under its weighted version, 
see e.g.~\citet[p.~277]{Lehmann1983}. In the simultaneous 
or multivariate problem, however, \citet{Peng1975} and 
\citet{CZ1975} were the first to show that $\delta_0$ 
\textit{can} be improved upon under the loss functions 
\beq
\label{eq:twolossfunctions}
L_0^*(\theta,\delta) = \sumip(\delta_i - \theta_i)^2
\quad{\rm and}\quad 
L_1^*(\theta,\delta) = \sumip(\delta_i - \theta_i)^2/\theta_i, 
\eeq 
if $p \geq 3$ and $p \geq 2$, respectively. 
In particular, for the $L_1^*$ loss, \citet{CZ1975} 
derived the estimator 
\beq
\label{eq:czestimator}
\delta_{\text{CZ},i}(Y) = \Bigl(1 - \frac{p-1}{p-1+Z} \Bigr)Y_i,
   \quad {\rm where\ }Z=\sumip Y_i, 
\eeq
demonstrating that it improves uniformly on 
the standard procedure $\delta_0$. 
This Stein-type phenomenon has also been observed for other 
loss functions. \citet{Hwang1982} obtained results for
$\sumip(\delta_i - \theta)^2/\theta_i^{m_i}$,  
for fixed integers $m_i$. \citet{GhoshYang1988} considered 
a loss function based on the entropy or Kullback--Leibler 
distance 
$L_e(\theta,\delta)=\sumip \theta_i
   \{\delta_i/\theta_i-\log(\delta_i/\theta_i)-1\}$,  
A good overview article for admissibility issues, 
for multivariate Poisson means and for other models 
for discrete data, is \cite*{GhoshHwangTsui1983}, 
followed by discussion contributions by 
\citet{Berger1983}, \citet{Morris1983b}, \citet{Hudson1983}. 
A more recent article on these issues is 
\cite*{BrownGreenshteinRitov2013}. 

More broadly, the books \citet{Efron2013} and 
\citet*{Shrinkage2018} contribute to seeing multivariate 
estimation, using shrinkage and empirical Bayes methods, 
as as a global phenomenon. Central themes are variations
on `borrowing strength', either via direct constructions
or in empirical Bayes setups. Our article is a contribution
in these general directions, showing that methods
developed for the multinormal and more generally
spherically symmetric distributions have certain
parallels in the world of multivariate Poisson 
estimation. Importantly, when estimating a particular
Poisson parameter $\theta_i$, the full multivariate
vector of data is being used, as an integral part 
of our methods. 

The estimators that have been found to be better than 
$\delta_0(Y) = Y$ in these earlier publications 
are essentially all of the shrinking type, pushing the
maximum likelihood estimator closer to the bottom corner 
of the parameter space. A good example of the merits of 
this type of shrinkage is provided by \citet{CZ1975}, 
wherein 36 small $\theta_i$ -- intensities of an oil-well 
discovery process -- are estimated with the estimator in 
\eqref{eq:czestimator}. 
\citet{CZ1975} had access to `known' $\theta_i$ 
and could check the actual loss incurred. The procedure 
in \eqref{eq:czestimator} did indeed give much smaller loss 
$\sum_{i=1}^{36}(\delta_{\text{CZ},i} - \theta_i)^2/\theta_i$ 
than did $\delta_0$; the losses are 14.33 and 39.26,
respectively. However, another and perhaps not so 
pleasant feature of their estimation procedure is 
conspicuous here, namely that the sum 
$\gamma = \sum_{i=1}^{36}\theta_i$ (or equivalently the 
mean $\bar{\theta} = \gamma/36$) is seriously underestimated. 
In their oil-well discovery example the true sum is 
$\gamma = 26.98$, the usual $\delta_0$ has 
$\sum_{i=1}^{36} y_i = 29$, while their 
$\sum_{i=1}^{36}\delta_{\text{CZ},i} = 12.97$ is much too low. 

In some situations the sum (or mean) is unimportant and 
all that matters is to estimate each $\theta_i$. In 
many multiparameter cases on wants to keep track of 
the sum (or the mean) of the $\theta_i$ as well, however, 
as one surely would in the oil-well discovery process 
above. Other multiparameter cases where the sum 
(or the mean) are deemed as important as the individual 
$\theta_i$ abound; think for example of a decision maker 
having to make budgetary decisions concerning each of 
the boroughs of a city \textit{and} the city as a whole, 
the resources allocated to each of $p$ hospitals and 
the whole health budget, etc. 

These considerations motivate studying loss functions 
that take into account the need for good individual 
estimates, while at the same time guarding against 
the underestimation of the sum. One example of 
such a loss function is
\beq
\label{eq:quadLoss1}
L(\theta,\delta) = \sumip (\delta_i - \theta_i)^2 
  + c \Bigl( \sumip\delta_i - \gamma\Bigr)^2
 = L_0^*(\theta,\delta)+ cp^2 (\bar{\delta} - \bar{\theta})^2,
\eeq
Since $\sumip\delta_{0,i}=\sumip Y_i$ is the admissible minimax 
solution under the loss function $\big(\sumip\delta_i - \gamma\big)^2$ 
one might wonder whether the extra penalisation above would 
secure admissibility of $\delta_0$ in the simultaneous problem. 
In Section \ref{section:quadLoss} we show that this is not 
the case; $\delta_0$ is again inadmissible when $p\geq 3$, 
under any given quadratic form loss function 
$(\delta-\theta)^\tr  A (\delta-\theta)$. Another loss function 
that takes the guarding against underestimation of the 
sum (or the mean) of the parameters into account is the 
weighted version of \eqref{eq:quadLoss1}, that is, 
the loss functions that generalises the one used 
by \citet{CZ1975}, namely
\beq
\label{eq:closs}
\begin{split}
L_c(\theta,\delta) 
   = \sumip\frac{(\delta_i - \theta_i )^2}{\theta_i} 
   +c{(\sumip\delta_i-\sumip\theta_i)^2\over \sumip\theta_i}
 = L_1^*(\theta,\delta) 
   + cp \frac{(\bar{\delta} - \bar{\theta})^2}{\bar{\theta}},
\end{split}
\eeq
where $c$ is a user-defined constant. Under \eqref{eq:closs}, 
the loss in Clevenson and Zidek's example for $\delta_0$ 
is $39.26 + 0.35\,c$ while the loss for $\delta_{\text{CZ}}$ 
is $14.33 + 6.35\,c$. Again, one might wonder whether 
$\delta_0$ is admissible for certain values of $c$. 
In Section \ref{section:closs} we show that this is not the case, 
and derive an estimator which is the natural generalisation 
of the one in \eqref{eq:czestimator}, 
\beq
\label{eq:emilnilsestimator}
\delta_i^c(Y) = \Bigl\{1 - \frac{p-1}{p-1+(1+c)Z}\Bigr\}\,Y_i
   \quad {\rm for\ }i=1,\ldots,p, 
\eeq  
which is shown to belong to a class of estimators that dominate 
$\delta_0$ uniformly over the parameter space. 
In the somewhat similar multivariate normal setting, 
investigations on how to limit the shrinkage of the 
James--Stein estimator, in order to account for objectives 
related to estimation the means, have been conducted by 
\citet{EfronMorris1971,EfronMorris1972}. 

Our paper proceeds as follows. In Section \ref{section:quadLoss} 
we study the loss function \eqref{eq:quadLoss1} and derive 
a class of estimators that uniformly dominate the maximum 
likelihood estimator. Section \ref{section:closs} concerns 
the weighted version of this loss function, the $L_c$ 
loss function of (\ref{eq:closs}), and we derive the already 
mentioned generalisation $\delta^c$ of $\delta_{\rm CZ}$. 
In Section \ref{section:Bayes} further light is shed 
on this new class of estimators by showing that it 
can be derived by Bayesian and empirical Bayesian methods. 
We are also able to prove that the estimator $\delta^c$ 
of (\ref{eq:emilnilsestimator}) is admissible. 
Classes of alternative estimators are then studied in 
Sections \ref{section:smoothtomean} and 
\ref{section:moreempiricalbayes}, involving 
Bayes and empirical Bayes strategies which 
shrink the raw data towards data-driven structures
for the $\theta_i$, such as the data mean, 
or a submodel. Some of these procedures succeed 
in having risks well below the minimax benchmark, 
in large regions of the parameter space, though without
achieving uniform dominance. In Section \ref{section:weightedloss} 
weighted loss functions are studied, which admit 
relative weights of importance; notably, an infinite-dimensional
setup is also included. Then in Section \ref{section:extra} 
we demonstrate how some of our Bayesian constructions 
also lead to new multivariate models for rate parameters 
and for count observations, of interest in their own right, 
pointing to models for spatially dependent count data. 
Finally Section \ref{section:concluding} 
offers a list of concluding remarks. 
 
\section{General quadratic loss function} 
\label{section:quadLoss}

We may write the loss function of \eqref{eq:quadLoss1} 
as $L(\delta,\theta) = (\delta - \theta)^\tr A_0(\delta - \theta)$, 
where $A_0$ is the matrix with $1+c$ down its diagonal 
and $c$ elsewhere. The natural generalisation is
\beq
L(\delta,\theta) = (\delta - \theta)^\tr A(\delta - \theta),
\label{eq:quadLoss1.r}
\eeq
where $A$ is symmetric and positive definite. Below we obtain 
some results for general $A$ and apply these to two examples. 
Note that in situations where there is no ordering of the 
individual $\theta_i$ and no reason to estimate some of 
them more precisely than the others, the loss function 
in \eqref{eq:quadLoss1}, that is, using $A_0$ with an 
appropriate choice of $c$, is the natural choice. Our 
method of proving inadmissibility of $\delta_0$ resembles 
that of \citet{TsuiPress1982} and \citet{Hwang1982}, 
where $A$ is diagonal. 

Let $\delta(y) = y - \phi(y)$ be a competitor. The difference 
in risk between these two estimators is then 
\beq
\begin{split}
R(\delta,\theta) - R(\delta_0,\theta) 
& = \E_\theta\,\{ (\delta - \theta)^\tr A(\delta - \theta ) 
   - (Y - \theta)^\tr A (Y - \theta)\} \\
& = \E_\theta\,\{ -2(Y - \theta) A \phi(Y) + \phi(Y)^\tr A\phi(Y)\} \\
& = \E_\theta\,\{ -2(Y - \theta)\psi(Y) + \psi(Y)^\tr A^{-1}\psi(Y)\},
\end{split}
\label{eq:riskdiff1}
\eeq
writing $\psi(y) = A \phi(y)$. Furthermore, 
since $\E_\theta\,\theta_i g(Y) = \E_\theta\,Y_i g(Y - e_i)$ for 
any $g$ with $\E_\theta\,|g(Y)|<\infty$, using $e_i$ to 
denote the unit vector with the $i$th element equal to one, 
we have that  
\beqn
\E_\theta\,(Y_i - \theta_i) \psi(Y) 
   = \E_\theta\,Y_i\{\psi(Y) - \psi(Y - e_i) \},
\eeqn 
and the risk difference \eqref{eq:riskdiff1} can be 
written $\E_\theta\,D(Y,\phi)$, in which 
\beqn
D(y,\phi) = -\sumip 2y_i\{ \psi(y) - \psi(y - e_i) \} 
   + \psi(y)^\tr A^{-1}\psi(y).
\eeqn 
If a function $\psi(y)$ can be found such that $D(y,\phi)$ 
is non-positive for all $y$, with strict inequality 
for at least one datum $y$, then $\delta_0 = Y$ is 
inadmissible, being outperformed by $\delta(y) = y - A^{-1}\psi(y)$.  

\begin{theorem}
\label{Th:quadloss}
Let $A$ be symmetric and positive definite. Then $\delta_0 = Y$ 
is inadmissible under loss function \eqref{eq:quadLoss1.r}, 
if $p\geq 3$. It is dominated by $\delta(Y) = Y - A^{-1}\psi(Y)$, 
where $\psi(\cdot)$ is any member of the following class 
\beqn
\psi_i(y) = \frac{d(y)}{B(y)}T(y_i), \quad \text{for\ }i = 1,\ldots p,
\eeqn 
where $T(0) = 0$ and $T(y) = \sum_{j\leq y} 1/j$ for $y \geq 1$, 
where $B(y) = \sumip T(y_i)T(y_i +1)$, and $d(y)$ is nondecreasing 
in each argument and obeying 
\beqn
0 \leq d(y) \leq (2/M)\{N(y) - 2  \}_{+},
\eeqn
writing $a_{+} =\max(a,0)$ for truncating to zero.  
Here $N(y) = \sumip I\{y_i \leq 1\}$ and $M$ 
is the inverse of the smallest eigenvalue of $A$. 
\end{theorem}

\begin{proof} 
The choice of $M$ entails $\psi^\tr A^{-1}\psi \leq M\sumip\psi_i^2$ so that
\beqn
D(\phi,y) \leq -2\sumip y_i\big\{\psi_i(y) -\psi_i(y-e_i) \big\} 
   + \sumip M\psi_i(y)^2.
\eeqn 
The general Theorem 2.1 in \citet{Hwang1982}, with accompanying 
corollaries, can then be used to find $\psi_i(\cdot)$ 
functions that make $D(\phi,y)$ non-positive for all $y$. 
We skip details but record that Hwang's method gives
\beq
D(\phi,y) \leq -2d(y)\frac{\big\{N(y) - 2 - Md(y)/2  \big\}_{+}}{B(y)},
\label{eq:Hwangdiffrisk}
\eeq
which is non-positive for each $\psi(\cdot)$ described in the theorem.
\end{proof}

A natural choice for $d(y)$ is the following, minimising 
the upper bound in \eqref{eq:Hwangdiffrisk}, 
\beqn
d_0(y) = (1/M)\{N(y) - 2 \}_{+}.
\eeqn
This means using $\phi_0(y) = A^{-1}\psi_0(y)$ with 
$D(\phi_0,y)\leq - M^{-1} \{ N(y) - 2 \}^2/B(y)$, 
which shows that 
$$ \delta(y)   = y - A^{-1}\psi_0(y) , \;\; 
   \text {where}\;\; 
\psi_{0,i}(y) = \frac{1}{M}\frac{T(y_i)}{B(x)}\{ N(y) - 2 \}_{+}, $$
achieves
$$ R(\delta,\theta) - R(Y,\theta) 
\leq - M^{-1}\E_\theta\,\big\{ (N(Y) - 2)_{+} \big\}^2/B(Y) < 0, $$
for all $\theta$. Note in particular that the same 
$\psi_0(\cdot)$ function works for a large class 
of loss function \eqref{eq:quadLoss1.r}. 

\begin{example}{\rm  
Let $A$ be the square matrix with $1+c$ down its diagonal 
and $c$ elsewhere. This matrix might be written $A = I + cee^\tr$, 
with $I$ the identity matrix and $e = (1,\ldots,1)^\tr$ 
the $p\times 1$ vector of ones. This choice of $A$ gives 
the loss function \eqref{eq:quadLoss1}.  Here $M= 1$ and 
$A^{-1} = I - \{c/(1+cp)\}ee^\tr$, and 
$\psi^\tr A^{-1}\psi \leq \sumip \psi_i^2$ follows. The natural 
estimator is then $\delta(y) = y - \phi_0(y)$, where
\beqn
\phi_0(y) = \psi_{i,0}(y) - \frac{c}{1+cp}\sum_{j=1}^p \psi_{0,j}(y) 
   = \{N(y)-2\}_{+}\frac{T(y_i) - \frac{cp}{1+cp}\bar{T}(y) }{B(y)},
\eeqn 
in which $\bar{T}(y) = p^{-1}\sumip T(y_i)$. Note that $\delta$ 
does not shrink the $y$ in any particular direction, but 
rather pushes the components $y_i$ in different directions 
according to the sign of $T(y_i) - \{cp/(1+cp)\}\bar T(y)$. 
Note further that if the $y_i$ are moderate or large, then 
$$\delta_i(y)\doteq y_i-{p-2\over \sum_{j=1}^p (\log y_j)^2}
  \Bigl\{\log y_i-{cp\over 1+cp}{1\over p}\sum_{j=1}^p\log y_j\Bigr\}. $$}
\end{example}

\begin{example}{\rm 
Samples of independent Poisson variables arise naturally
when one or more Poisson processes are observed over time.
Dividing the time interval into $p$ parts gives counts 
$y_1,\ldots,y_p$ with certain means $\theta_1,\ldots,\theta_p$. 
In the nonparametric setting, where the intensity of the process
is unknown, these parameters are also completely unknown. 
If one wishes to estimate the cumulative intensity 
of the process, then $(\lambda_1,\ldots,\lambda_p)$ are more 
important than $(\theta_1,\ldots,\theta_p)$, where 
$\lambda_i=\sum_{j=1}^i \theta_j$.
This suggests using the loss function 
$L(\theta,\delta)=\sumip (\hatt\lambda_i-\lambda_i)^2$, 
where $\hatt\lambda_i=\sum_{j=1}^i \delta_j$. 
But this is seen to be a special case of (\ref{eq:quadLoss1.r}), 
with elements $a_{i,j}=p+1-\max(i,j)$ filling the $A$ matrix. 
Its inverse $A^{-1}$ has first row $(1,-1,0,\ldots,0)$, 
last row $(0,\ldots,0,-1,2)$, and in between we find 
$(0,\ldots,0,-1,2,-1,0,\ldots,0)$. One has 
$\psi^\tr A^{-1}\psi\le 4\psi^\tr\psi$, and 
can use $M=4$ when applying the theorem. 
Hence the following estimator improves on $\delta_0=Y$:
\beqn 
\delta_1^*(y)&=&y_1-\psi_{0,1}(y)+\psi_{0,2}(y), \cr
\delta_i^*(y)&=&y_i+\psi_{0,i-1}(y)-2\psi_{0,i}(y)+\psi_{0,i+1}(y)
        \quad {\rm for\ }2\le i \le p-1, \cr
\delta_p^*(y)&=&y_p+\psi_{0,p-1}(y)-2\psi_{0,p}(y),
\eeqn 
where this time 
$$\psi_{0,i}(y)={1\over 4}{T(y_i)\over B(y)}\{N(y)-2\}_+. $$
Notice finally that the corresponding improved
estimators for the cumulative $\lambda_i$ become 
$$\hatt\lambda_i=\sum_{j=1}^i Y_j
                -\psi_{0,i}(Y)+\psi_{0,i+1}(Y) \quad {\rm for\ }i\le p-1, $$
while $\hatt\lambda_p=\sum_{j=1}^p Y_j-\psi_{0,p}(Y)$.} 
\end{example}


\section{The $L_c$ loss function}
\label{section:closs}

The main consideration leading to the loss function 
$(\delta_i - \theta_i)^2/\theta_i$ is that the statistician 
seeks precise estimates of small values of $\theta_i$. 
A loss function that penalises heavily for bad estimates 
of small parameters is then a natural choice. Related to 
this is the obvious fact that when the parameters are small, 
they can only be badly overestimated, zero being the boundary 
of the parameter space. The corresponding multiparameter 
version of this is $\sumip(\delta_i - \theta_i)^2/\theta_i$, 
and is the one treated by \citet*{GhoshHwangTsui1983}, 
\citet{Hwang1982}, \citet{TsuiPress1982}, \citet{CZ1975}, 
and others. Note that $\delta_0 = Y$ has constant risk $p$ 
with this loss, and it is not difficult to establish that 
it is minimax. 

The above mentioned authors obtain classes of estimators 
that perform uniformly better than $\delta_0$ if only 
$p \geq 2$. As discussed in Section \ref{section:intro} and 
illustrated by the oil-well example, these shrinkage 
estimators do not take into account the additional 
desideratum, namely a precise estimate of the sum 
$\gamma = \sumip\theta_i$. We will now study the 
$L_c$ loss function of \eqref{eq:closs}, 
$L_c(\theta,\delta) = L_1^*(\theta,\delta) 
   + c(\sumip\delta_i - \gamma)^2/\gamma$. 
If we consider the second term in \eqref{eq:closs} 
by itself, we recognise the one-dimensional loss function 
$(\delta - \gamma)^2/\gamma$. It is well known that 
$\hatt\gamma=Z$ is admissible and the unique minimax solution 
under this loss function, and can therefore not be uniformly improved 
upon \cite[p.~277]{Lehmann1983}. Consequently, since $z$ 
is the sum of the $y_i$, higher values of $c$ will result
in estimators that lie closer the $\delta_0$. On the other hand, 
we know that for $c=0$ the estimator in \eqref{eq:czestimator} 
uniformly dominates $\delta_0$. Hence, the user defined constant 
$c$ determines how to compromise between $\delta_0$ and 
$\delta_{\text{CZ}}$.  

Before we derive a class of estimators that dominate $\delta_0$ 
under $L_c$ in Section \ref{section:dominatingclass}, we derive 
formulae for the Bayes solution and show that $\delta_0$ 
is minimax.




\subsection{The $\delta_0=Y$ estimator is minimax.} 

The maximum likelihood estimator $\delta_0=Y$ has constant risk 
$p+c$ under $L_c$ and is a natural candidate for being minimax. 
We demonstrate minimaxity by exhibiting a sequence of priors 
with minimum Bayes risks $\text{BR}(\delta,\pi) = \E\,R_c(\delta,\theta)$ 
which converge towards $p+c$; 
that this is sufficient follows from well-known arguments,
as exposited e.g.~in \citet[Ch.~2.4]{Robert2007}. 
Some analysis is required to characterise the Bayes solution. 
We first find the values $\delta_1,\ldots,\delta_p$ that 
minimise the posterior expected loss, i.e.~given a dataset
$y=(y_1,\ldots,y_p)$, with respect to some distribution 
over the parameter space. With $\gamma=\sumip\theta_i$, 
introduce  
$$ a_i =  \{\E(\theta_i^{-1}\midd y)\}^{-1},\quad
   a = \sumip a_i,\quad \text{and}\quad  
   b = \{\E(\gamma^{-1}\midd y)\}^{-1}.$$
Then 
\beqn
\E\,\{L_c(\theta,\delta)\midd y\}  
   = \sumip \big\{\delta_i^2/a_i - 2\delta_i + \E(\theta_i\midd y) \big\}
+ c\Bigl\{ \Bigl(\sumip\delta_i \Bigr)^2/b - 2\sumip \delta_i 
   + \E(\gamma\midd y)  \Bigr\},
\eeqn
assuming the moments to exist. Some analysis shows that 
the minimum takes place for 
\beq
\label{eq:bayes1}
\delta_i^B(y) 
   = \frac{1+c}{1+ca/b}a_i
   = \frac{1+c}{1+ca/b}{1\over \E\,\{(1/\theta_i)\midd y\}}
\quad {\rm for\ }i=1,\ldots,p. 
\eeq
This generalises the familiar result that the Bayes solution 
is $\{\E(\theta_i^{-1}\midd Y)\}^{-1} = a_i$ 
under $L_1^*(\theta,\delta)$, that is, when $c=0$. 
Note also that if $\E\,\theta_i^{-1} = \infty$ for some $i$, 
and $\E\,\theta_i$ is finite, then only $\delta_i = 0$ 
gives a finite risk, which means that \eqref{eq:bayes1} 
is correct even in such cases.   

To illustrate this, suppose $\theta_i$ has a Gamma prior 
with parameters $(\alpha_i,\beta)$, 
which we write as $\Gam(\alpha_i,\beta)$, 
i.e.~with prior mean $\alpha_i/\beta$, 
and that these are independent. Then 
$\theta_i\midd y\sim\Gam(\alpha_i+y_i,\beta+1)$, 
and some calculations lead to the Bayes estimators
\beqn
\begin{array}{rcl}
\delta_i^B
&=&\displaystyle 
{1+c\over 1+c(p\bar\alpha+z-p)/(p\bar\alpha+z-1)}
   {\alpha_i+y_i-1\over \beta+1} \\ 
&=&\displaystyle 
{(1+c)(p\bar\alpha+z-1)\over 
   (1+c)(p\bar\alpha+z-1)-c(p-1)}
      {\alpha_i+y_i-1\over \beta+1} 
\end{array}
\eeqn 
for $i=1,\ldots,p$, writing $\bar\alpha=(1/p)\sumip\alpha_i$
and again $z=\sumip y_i$. 
If $\alpha_i\le1$, then the Bayes estimate is zero 
if $y_i=0$, by the comment above about (\ref{eq:bayes1}).
We note that for large $c$, corresponding to the loss
being essentially related to estimating the 
sum $\gamma=\sumip\theta_i$ well under 
$(\hatt\gamma-\gamma)^2/\gamma$ loss, 
\beqn
\delta_i^B\doteq {p\bar\alpha+z-1\over \beta+1}
   {\alpha_i+y_i-1\over p\bar\alpha+z-p}, 
   \quad {\rm with\ sum} \quad 
   {p\bar\alpha+z-1\over \beta+1},
\eeqn 
tying in with $\gamma\midd y$ being a Gamma 
with parameters $(p\bar\alpha+z,\beta+1)$. 

Now consider the special case where the $\theta_i$
are independent Gammas with parameters $(1,\beta)$.
The Bayes solution then takes the form 
\beq
\delta_i^B 
   = \frac{1+c}{1+cz/(p-1+z)}\frac{y_i}{1+\beta} 
   = \frac{(1+c)(p-1+z)}{p-1+(1+c)z}\frac{y_i}{1+\beta}.
\label{eq:gammaBayes}
\eeq
It now remains to show that the minimum Bayes risk 
for this prior, say $\text{MBR}(1,\beta)$, 
tends to $p+c$ as $\beta \to 0$. Using that $Y_i$ given $Z$ 
is binomial with mean $Z\theta_i/\gamma$, 
provided $Z\ge1$, and that the $Y$ vector and hence 
the $\delta^B$ estimator are equal to zero when $Z=0$,   
the risk of $\delta_i^{B}$ can be expressed as 
\beqn
\begin{split}
R_c(\delta_i^{B},\theta)  
   & = \E_\theta\,\E_\theta\,\{L(\theta,\delta_i^{B})\midd Z\} \\
   & = \E_{\gamma} \Bigl\{  \frac{(1+c)^2(p-1+Z)^2}{p-1+(1+c)Z}
   \frac{Z}{\gamma\,(1+\beta)^2} \\
&\hspace{2cm} - 2\frac{(1+c)^2(p-1+Z)}{p-1+(1+c)Z}
   \frac{Z}{1+\beta} + (1+c)\gamma \Bigr\} .
\end{split}
\notag
\eeqn
Since the risk depends on the $\theta_i$ only through 
the sum $\gamma$, the minimum Bayes risk may be written 
\beq
\text{MBR}(1,\beta)  
   = \E_\gamma\,\E_\gamma\,\{L(\theta,\delta^{B})\midd Z\}
   =  \frac{1+c}{1+\beta}\,\E_\gamma\,\frac{p(p-1) + (p+c)Z}{p-1+(1+c)Z}, 
\notag
\eeq
in which the expectation on the right is with respect 
to the marginal distribution of $Z$. Since $Z$ tends 
in probability to infinity as $\beta \to 0$, 
the function above converges in probability, 
\beqn
\frac{p(p-1) + (p+c)Z}{p-1+(1+c)Z} \rightarrow_\pr \frac{p+c}{1+c},
\eeqn
when $\beta \to 0$. Furthermore, this function 
is bounded by $p$, so by the bounded convergence 
theorem $\text{MBR}(1,\beta)$ tends to $p+c$ as $\beta$ 
goes to zero, as was to be shown.  

\subsection{A dominating class of estimators.}
\label{section:dominatingclass}

We will now develop a class of estimators with uniformly 
smaller risk than $p+c$ under the $L_c$ loss function, 
that is, estimators that uniformly dominate the maximum 
likelihood estimator. Consider estimators of the form 
$\delta_i(Y) = \{1 - \phi(Z)\}Y_i$. 
Write $D(\phi,y) = L_c(\theta,\delta) -  L_c(\theta,\delta_0)$ 
for the difference in loss. Then 
\beqn
D(\phi,y)
= \sumip\Bigl\{ \frac{\phi(z)^2y_i^2 - 2\phi(z)y_i(y_i - \theta_i)}{\theta_i}
+ c \frac{\phi(z)^2z^2 -2 \phi(z)z(z-\gamma)}{\gamma}\Bigr\},
\eeqn 
Now, use the fact that for any real valued function $h$
with finite mean $\E_\theta\,h(Y)$, and with the property 
that $h(y) = 0$ whenever $y_i = 0$,  
the following identity holds: 
\beq
\label{eq:useful.identity}
\E_\theta\,h(Y)/\theta_i = \E_\theta\,h(Y + e_i)/(Y_i + 1).
\eeq
Using this identity we obtain an expression for the 
difference in risk 
$\E_\theta\,D(\phi,Y) = R(\delta^*,\theta) -R(Y,\theta)$, namely  
\beq
\begin{split}
\E_\theta&D(\phi,Y) = \E_\theta\,\E_\theta\,\{D(\phi,Y)\midd Z\} \\ 
& = \E_\theta\,\Bigl[\{\phi^2(Z)-2\phi(Z)\}\frac{Z\{(p-1) + (1+c)Z\}}{\gamma} 
   + 2(1+c)\phi(Z)Z\Bigr] \\ 
& = \E_\theta\,\big\{(\phi^2(Z+1)-2\phi(Z+1)) \{(p-1) + (1+c)(Z+1)\}
 + 2(1+c)\phi(Z)Z\big\}. 
\end{split}
\notag
\eeq
This can hence be expressed as $\E_\gamma\,D^*(\phi,Z)$,
with the $D^*(\phi,z)$ function not depending on the 
parameters; in particular, the risk function 
$R_c(\delta,\theta)=p+c+\E_\gamma\,D^*(\phi,Z)$ 
depends on the parameter vector only via 
$\gamma=\sumip\theta_i$. Also, any function $\phi(\cdot)$ 
that ensures that $D^*(\phi,z)\leq 0$ for all $z$, 
with strict inequality for at least one datum $z$, 
yields an estimator that uniformly dominates the $\delta_0$. 
This leads to the following result.

\begin{theorem}
\label{th:delta_c_theorem} 
For each function $\phi(\cdot)$ such that
\beqn
0 < \phi(z) < \frac{2(p-1)}{p-1+(1+c)z} 
   \quad\text{and}\quad
   \phi(z)z < \phi(z+1)(z+1) 
\eeqn
for all $z$, the estimator $\delta_c = \{1 - \phi(Z)\}Y$ 
uniformly dominates $\delta_0 = Y$. These conditions 
are met for functions of the type $\phi(z) = \psi(z)/\{p-1+(1+c)z\}$ 
where $\psi(\cdot)$ is nondecreasing with $0 < \psi(z)  < 2(p-1)$. 
In particular, $\delta_0$ is inadmissible if $p \geq 2$.  
\end{theorem} 

\begin{proof} 
Using the expression for $D^*(\phi,Z)$ derived above, 
the following holds: 
\beqn
\E_\gamma\,D^*(\phi,Z)
&=&\E_\gamma \Bigl[\{\phi(Z)^2-2\phi(Z)\}{Z(p-1)+(1+c)Z\over \gamma}
   +2(1+c)\phi(Z)Z\Bigr] \\
&=&\E_\gamma\,\phi(Z)Z \Bigl[\{\phi(Z)-2\}{p-1+(1+c)Z\over \gamma}
   +2(1+c)\Bigr] \\
&=&{1\over\gamma}\E_\gamma\,\phi(Z)Z \bigl[\phi(Z)\{p-1-(1+c)Z\}
   -2 \{p-1+(1+c)(Z-\gamma)\}\bigr] \\
&\le&{1\over\gamma}\E_\gamma\,\phi(Z)Z 
   \bigl[2(p-1)-2\{p-1+(1+c)(Z-\gamma)\}\bigr] \\
&=&-{1+c\over \gamma}\E_\gamma\, \phi(Z)Z(Z-\gamma) \\
&=&-(1+c)\,\E_\gamma\,\{\phi(Z)Z^2/\gamma-\phi(Z)Z\} \\
&=&-(1+c)\,\E_\gamma\,\{\phi(Z+1)(Z+1)-\phi(Z)Z\} < 0. 
\eeqn 
This is valid for all $\gamma$ since $\phi(z)z$ 
is a strictly increasing function of $z$. 
\end{proof}

We denote by $\mathcal{D}_c$ the class of estimators 
\beqn
\{1 - \phi(Z)\}Y \quad{\rm where}\quad 
   \phi(z) = \psi(z)/\{p-1+(1+c)z\} 
\eeqn 
and with $\psi(\cdot)$ satisfying the conditions of 
Theorem \ref{th:delta_c_theorem}. The optimal choice 
of $\psi(\cdot)$ in terms of minimising risk, based 
on the simple upper bound of $\E_\theta\,D^*(\phi,Z)$, 
is $\psi(z) = p-1$, leading to the estimator
\beq
\label{eq:delta1}
\delta_i^c(Y) = \Bigl\{1 - \frac{p-1}{p-1+(1+c)Z}\Bigr\} Y_i
   \quad {\rm for\ }i=1,\ldots,p. 
\eeq 
Note that $\delta^c$ appropriately generalises the 
Clevenson--Zidek estimator of \eqref{eq:czestimator}. 
Importantly, it is clearly seen how fine-tuning of 
$c$ determines the amount of shrinkage away from 
the $\delta_0$. We can use the expression for $D^*(\phi,z)$ 
derived above to find the risk function for the 
estimator in \eqref{eq:delta1}, 
\beq
R_c(\delta^c,\theta)  
   = p + c - \E_\gamma\, \frac{(p-1)^2}{p-1+(1+c)(Z+1)}
   \Bigl\{1 + \frac{2(1+c)}{p-1+(1+c)Z}\Bigr\}.
\notag
\eeq
Note that the risk depends on $\theta_i$ only through 
the sum $\gamma$, and that numerical evaluation is easy 
because $Z$ is Poisson with mean $\gamma$. The risk function 
starts at 
\beqn
R_c(\delta^c,\theta)
   = p+c - \frac{(p-1)^2}{p+c} - 2(1+c)\frac{p-1}{p+c} 
   = \frac{(1+c)^2}{p+c},
\eeqn 
for $\theta=0$, and then increases continuously towards 
the minimax risk $p+c$. As illustrated in Figure \ref{figure:risk1}, 
the improvement over $\delta_0=Y$ is substantial 
for small to moderate values of $\gamma$, 
and always lies below the risk of the usual estimator $\delta_0$.   

\begin{figure}[ht]
\centering
\includegraphics[scale=0.55,angle=0]{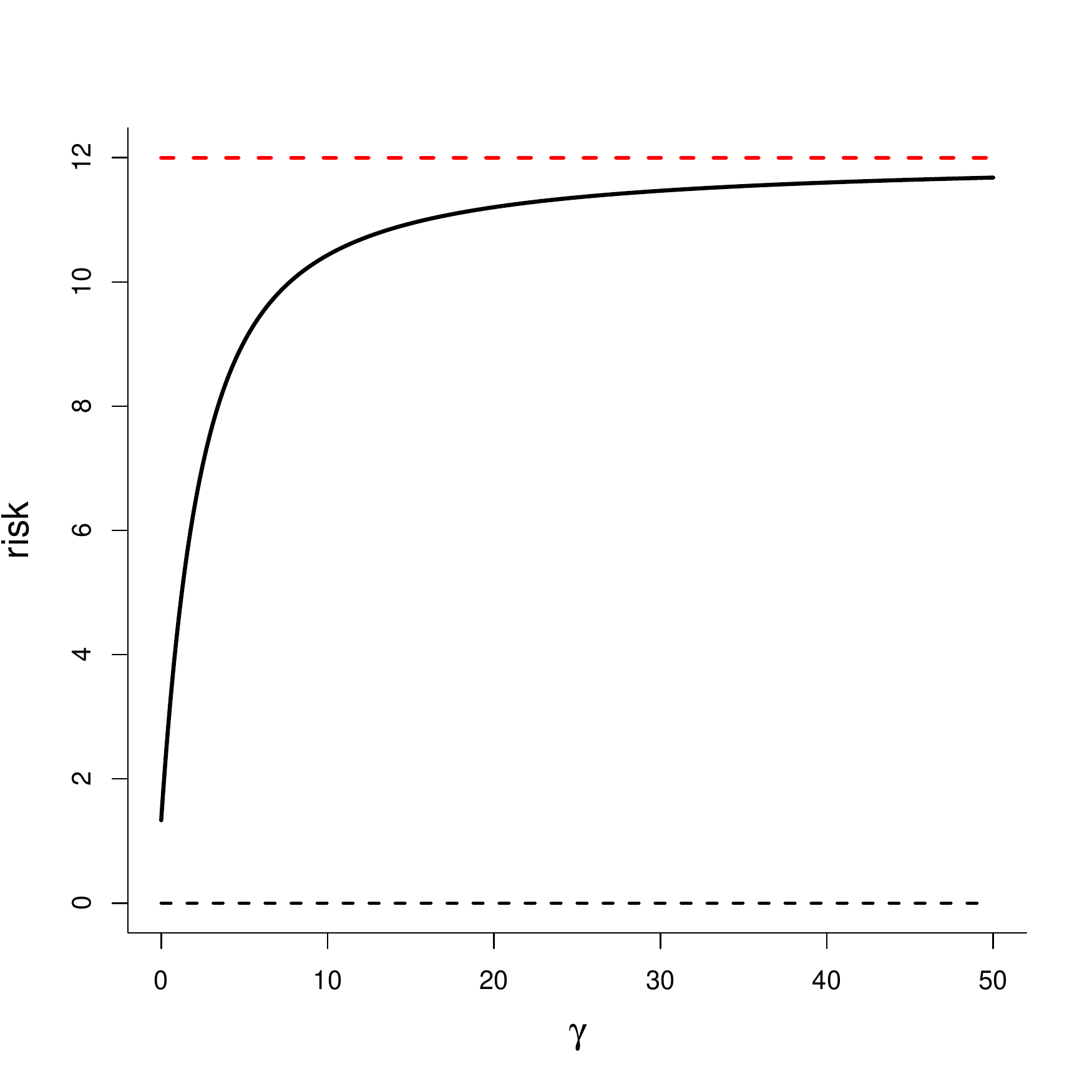}
\caption{The risk functions of $\delta^c$ of \eqref{eq:delta1} 
(full curve) and of $\delta_0 = Y$ (constant, slanted), 
with $p = 9$ and $c = 3$.}
\label{figure:risk1}
\end{figure}

\subsection{Loss function robustness.}

Robustness of performance statements with respect to 
the specific loss function used is often overlooked
in the literature, as if the loss function worked 
with had been handed down from above with absolute
precision. The matter is clearly of importance, 
however, as e.g.~briefly pointed to in comments by 
both \citet{Berger1983} and \citet{Morris1983b}. 
If an estimator performs well with respect to 
one loss function ${\rm Lo}_a$, but not for
another loss function ${\rm Lo}_b$, say, even
when these two are close, it is a cause for concern. 
We argue here, however, that our estimator (\ref{eq:delta1}), 
derived under loss function $L_c$ of (\ref{eq:quadLoss1}), 
is somewhat robust regarding the precise value of $c$. 

First consider the estimator $\delta_{\text{CZ}}$ 
of \eqref{eq:czestimator}, which uses $\phi_0(z) = (p-1)/\{p-1+z\}$. 
It satisfies the second requirement of 
Theorem \ref{th:delta_c_theorem}, 
i.e.~$\phi_0(z)z < \phi_0(z+1)(z+1)$ for all $z$, 
but it does not necessarily satisfy the first requirement, 
namely $0 < \phi(z) \leq 2(p-1)/\{p-1+(1+c)z\}$ for all $z$. 
It is easy to verify that if $0 \leq c \leq 1$ 
then it does satisfy this requirement, however, 
showing that $\delta_{\text{CZ}}$ has certain robustness 
properties with respect to the $L_c$ loss function: 
It is minimax and uniformly dominates $\delta_0$ under $L_c$, 
provided $0 \leq c \leq 1$.        

One can similarly study how the 
$\delta^{c_0}(Y) = [(1+c_0)Z/\{p-1+(1+c_0)Z\}]\,Y$ 
of \eqref{eq:delta1} fares when the loss function is 
not quite the $L_{c_0}$ under which it was derived, 
but rather $L_c$, with another penalty value of $c$, 
that is, with a somewhat different penalty paid to 
incorrect estimation of the sum $\gamma$. From the first 
condition of Theorem \ref{th:delta_c_theorem} we see that 
$\delta^{c_0}$ is still minimax and uniformly better than 
$\delta_0 = y$ under $L_c$, provided that 
$ 0 \leq c \leq 2c_0 + 1$. An immediate implication 
of this is that all estimators in the class $\mathcal{D}_c$ 
are minimax and uniformly better than $\delta_0$ under 
the $L_1^*$ loss function, showing that the more 
prudent estimation strategy $\delta^c$, in the sense 
that it shrinks less that $\delta_{\text{CZ}}$, is robust to $c=0$.    

\section{Bayes, empirical Bayes, and admissibility} 
\label{section:Bayes}

In this section a certain class of priors is studied 
along with Bayes and empirical Bayes consequences. 
The estimator $\delta^c$ of \eqref{eq:delta1} will 
be shown to be both a natural generalised Bayes estimator 
with respect to a certain noninformative prior, and
a natural empirical Bayes estimator with respect to 
independent Gamma priors. In addition, a class of 
proper Bayes estimators belonging to the class 
$\mathcal{D}_c$ is derived from another prior 
construction. Finally, we are also able to show that 
for each given $c$, the $\delta^c$ of \eqref{eq:delta1} 
is an admissible estimator, under the $L_c$ loss function.
In particular, the Clevenson--Zidek estimator (\ref{eq:czestimator})
is admissible under the $L_0$ function for which
it was derived. Our result hence generalises
that of \citet[Theorem 4.1]{Johnstone1984b}. 

\subsection{Priors with sum independent of proportions.}

In the following we model the means 
$(\theta_1,\ldots,\theta_p)$ in terms of the sum 
$\gamma=\sumip\theta_i$ and proportions 
$\pi_i = \theta_i/\gamma$ for $i=1,\ldots,p$. 

\begin{lemma} 
\label{lemma:lemma1}
Suppose that $\gamma$ and $\pi = (\pi_1,\ldots, \pi_p)$ 
are independent with simultaneous prior density 
$q(\gamma)r(\pi_1,\ldots, \pi_{p-1})$. 
This corresponds to a density 
$$q\Bigl(\sumip\theta_i\Bigr)\frac{1}{(\sumip\theta_i)^{p-1}}\, 
   r\left(\frac{\theta_1}{\sumip\theta_i},\ldots,
   \frac{\theta_{p-1}}{\sumip\theta_i}\right) $$
for $(\theta_1,\ldots,\theta_p)$. The posterior distribution for 
$\theta$ given the data is on the same form: $\gamma$ and 
$(\pi_1,\ldots,\pi_p)$ are still independent, and
\beqn
\begin{split}
\gamma\midd y &\sim \text{\normalfont const.} 
   \, \gamma^z e^{-\gamma} q(\gamma),\\
\pi\midd y &\sim \text{\normalfont const.} 
   \,\pi_1^{y_1}\cdots \pi_{p-1}^{y_{p-1}} 
   (1 - \pi_1 - \cdots - \pi_{p-1})^{y_p} r(\pi_1,\ldots,\pi_{p-1}). 
\end{split}
\eeqn
\end{lemma}

\begin{proof} 
The first part in an exercise in the transformation of 
random variables, involving calculating the determinant 
of Jacobi matrix 
$|\partial(\gamma,\pi_1,\ldots,\pi_{p-1})/\partial(\theta_1,\ldots,\theta_p)|
   =  \big(\sumip\theta_i\big)^{-(p-1)}$. The second part 
follows because the combined Poisson likelihood 
$\prod_{i=1}^p \theta_i^{y_i} e^{-\theta_i}/y_i!$ 
is proportional to 
$\gamma^{z} e^{-\gamma} \pi_1^{y_1}\cdots\pi_{p-1}^{y_{p-1}}
   (1-\pi_1-\cdots-\pi_{p-1})^{y_p}$. 
\end{proof}

If in particular $(\pi_1,\ldots,\pi_p)$ has a Dirichlet prior 
distribution with parameters $(\alpha_1,\ldots,\alpha_p)$, 
then the posterior is another Dirichlet with updated parameters 
$(\alpha_1 + y_1,\ldots,\alpha_p + y_p)$, and this holds 
regardless of the prior used for $\gamma$. The important 
case of independent $\theta_i$ from a Gamma $(\alpha,\beta)$ 
corresponds to a Gamma prior $(p\alpha,\beta)$ for $\gamma$ 
and an independent Dirichlet $(\alpha,\ldots,\alpha)$ for the proportions.  

Suppose $(\pi_1,\ldots,\pi_p)$ comes from a symmetric 
Dirichlet $(\alpha,\ldots,\alpha)$ independent of $\gamma$, 
the latter coming from a suitable prior $q(\cdot)$. 
The Bayes estimator under the $L_c$ loss function 
takes the form \eqref{eq:bayes1}, 
i.e.~$\delta_i = \{(1+c)/(1 + ca/b)\}a_i$, with
\beq
a_i = \big\{\E\,(\gamma^{-1}\pi_i^{-1}\midd y) \big\}^{-1}
= \big\{\E\,(\gamma^{-1}\midd y) \big\}^{-1}
   \big\{\E\,(\pi_i^{-1}\midd y) \big\}^{-1}
 = \frac{K(z)}{K(z-1)}\frac{\alpha + y_i - 1}{p\alpha + z - 1},
\notag
\eeq 
and $b = \big\{\E\,(\gamma^{-1}\midd Y) \big\}^{-1} = K(z)/K(z-1)$, 
writing $K(z) = \int_0^{\infty} \gamma^z e^{-\gamma}q(\gamma)\,\dd\gamma$. 
Letting in particular $\pi$ be uniform over the simplex 
we obtain the estimator 
\beq
\begin{split}
\delta_{B,i}(y) &= \frac{1+c}{1+ca/b}\,a_i 
= \frac{K(z)}{K(z-1)}\frac{1+c}{p-1+(1+c)z}y_i.
\end{split}
\label{eq:Bayes.symmetric}
\eeq
We are now in a position to give three pleasing interpretations 
of the estimator $\delta^c$ of \eqref{eq:delta1}. 

First, it is a \textit{generalised Bayes estimator}. For if 
$\gamma$ is given a flat prior $q(\gamma) = 1$ on the halfline, 
then $K(z) = \Gamma(z+1)$, which inserted in \eqref{eq:Bayes.symmetric} 
gives $\delta_B = \delta^c$. 

Second, it is a \textit{limit of Bayes estimators}. 
Let the proportions $(\pi_1,\ldots,\pi_p)$
have a flat Dirichlet $(1,\ldots,1)$ prior, 
with an independent $\gamma$ from a Gamma $(1,\beta)$. 
The Bayes solution is then 
\beqn
\delta_i^B = \bigg( 1 - \frac{p-1}{p-1+(1+c)z} \bigg) \frac{y_i}{1 + \beta},   
\eeqn 
and the limit as $\beta \to 0$ is again $\delta_i^c$ 
of \eqref{eq:delta1}. 

Thirdly, it is a natural \textit{empirical Bayes estimator}. 
To see one of several such constructions let the $\theta_i$ 
be independent $\Gam(1,\beta)$. The corresponding 
exact Bayes solution is given in \eqref{eq:gammaBayes}. 
Now $\beta$ is to be estimated from the data. The marginal 
distribution of $Z$ is found from the facts that $Z$ given $\gamma$ 
is Poisson with mean $\gamma$, and $\gamma$ 
is Gamma distributed with parameters $(p,\beta)$, so
\beqn
\text{Pr}\{Z= z\} = \frac{\beta^p}{\Gamma(p)}\frac{1}{z!} 
\frac{ \Gamma(p+z)}{(1+\beta)^{p+z}},\;\; z = 0,1,2,\ldots.
\eeqn
The sum $Z$ is sufficient for $\beta$ and $Z/(p-1+Z)$ 
is found to be an unbiased estimator for $1/(1+\beta)$,
for any $p\ge2$.  
Inserting this data-based value in \eqref{eq:gammaBayes} 
produces once again $\delta^c$. 

Finally, the Gamma $(1,\beta)$ prior construction may 
be extended to an hierarchical setup where a prior 
is put on the hyperparameter $\beta$. This extends 
the approach of \citet{GhoshParsian1981} 
from the $L_1^*$ of (\ref{eq:twolossfunctions}) to the $L_c$ setting.
Let the parameter $\beta$ be distributed according to 
\beq
\label{eq:hyperprior}
s(\beta) \propto \beta^{\eta-1}(\beta+1)^{-(\eta+\zeta)}.
\eeq
Utilising that 
\beqn
\E\,\Bigl({1\over \theta_i}\midd y\Big)
   =\E\,\E\,\Bigl({1\over \theta_i}\midd y,\beta\Bigr)
  =\E\,\Bigl({1+\beta\over y_i}\midd y\Bigr), 
\eeqn 
the Bayes solution is then
\beq
\delta_i^B(y) = \frac{(1+c)(p-1+z)}{p-1 + (1+c)z} \frac{y_i}
   {\E\,(1 + \beta \mid z)},
\label{eq:gamma.post}
\eeq
since $Z$ is sufficient for $\beta$. 
Also, $\beta$ given $Z=z$ is distributed as 
\beqn
\beta\midd z \sim \{B(\eta + p,\zeta + z) \}^{-1} 
   \beta^{\eta + p - 1}(1+\beta)^{-(\eta + \zeta + z + p)},
\eeqn
where $B(a,b) = \Gamma(a)\Gamma(b)/\Gamma(a+b)$ is 
the Beta function. Provided $\beta$ comes from the 
$s(\beta)$ of \eqref{eq:hyperprior}, and 
$0 < \eta \leq (p-2-c)/(1+c)$, then the Bayes solution 
in \eqref{eq:gamma.post}, by virtue of belonging 
to the class $\mathcal{D}_c$, is minimax 
and uniformly dominates $\delta_0$. To see this, insert 
\beqn
\E\,(1+\beta\midd z) = \frac{p+\eta+\zeta+z-1}{z+\zeta-1} 
\eeqn
in \eqref{eq:gamma.post}, yielding   
\beq
\delta_i^B(y) = \frac{(1+c)(p-1+z)}{p-1+(1+c)z}
   \frac{z+\zeta-1}{p+\eta+\zeta+z-1}y_i.
\label{eq:minimax.Bayes.solution}
\eeq
By some algebra we obtain that for the Bayes solution we here consider 
\beq
\begin{split}
\psi(z)  
   & = p-1+(1+c)z - (1+c)(p-1+z)\frac{z+\zeta-1}{p+\eta+\zeta+z-1}\\
   & = (1+c)\frac{(p-1+z)(p+\eta)}{p-1+\eta+\zeta+z} - c(p-1) .
\end{split}
\notag
\eeq
This function is non-decreasing for all $z\geq 0$. 
Moreover, we see that it is bounded above by 
\beq
\sup_{z\geq 0 }\psi(z) = (1+c)(p+\eta) \leq 2(p-1),
\notag
\eeq
since $\eta \leq (p-2-c)/(1+c)$. This means that the 
class of Bayes solutions in \eqref{eq:minimax.Bayes.solution}, 
with $\eta$ obeying the constraint mentioned, 
satisfy both conditions of Theorem~\ref{th:delta_c_theorem}.  

\subsection{Admissibility.}
\label{section:admissibility}

So far we have studied estimators that dominate the 
maximum likelihood estimator under the $L_c$ loss function. 
In this section we will prove that the estimator 
$\delta^c$ of \eqref{eq:delta1} cannot be uniformly 
improved upon, that is, $\delta^c$ is admissible. 
The basic ingredient in this proof is the characterisation 
of admissibility given by \citet[Theorem 2.6]{BrownFarrel1988}. 
According to this theorem an estimator $\delta$ is 
admissible for $\theta$ under $(\delta - \theta)^2/\theta$ 
if and only if there exists a sequence of finite measures 
$\nu_n$ such that the Bayes solution with respect to 
$\nu_n$, say $\delta_n$, converges to $\delta$ 
as $n\to \infty$, and
\beqn
\lim_{n\to\infty}\{\text{BR}(\delta,\nu_n) - \text{MBR}(\nu_n)\} = 0.
\eeqn 
This prior sequence has to satisfy certain conditions, 
the details of which are stated in \citet{BrownFarrel1988} 
and \citet[pages 237-238]{Johnstone1984b}. 
For our purpose, it is sufficient to know that 
such a sequence exists if $\delta$ is admissible. 

Consider the class of estimators given by 
\beq
\label{eq:symform}
\delta_i = \frac{(1+c)\kappa(Z)}{p-1+(1+c)Z}\,Y_i
         =\frac{(1+c)Y_i}{p-1+(1+c)Z}\,\kappa(Z)
   \quad {\rm for\ }i=1,\ldots,p.
\eeq
The Bayes estimators of \eqref{eq:Bayes.symmetric} are 
on this form with $\kappa(z) = K(z)/K(z-1)$; 
in particular the estimator $\delta^c$ of \eqref{eq:delta1} 
is on this form, with $\kappa(z)  = z$. As in 
\citet{Johnstone1984b}, it turns out that estimators 
of the class \eqref{eq:symform} are admissible 
provided that $\kappa(Z)$ is admissible for $\gamma$ 
under the loss function $(\kappa-\gamma)^2/\gamma$.  
The theorem below is in part a restatement 
of his Theorem 4.1. 

\begin{theorem}
\label{th:admissibility} 
Estimators of the form \eqref{eq:symform} are admissible 
for $(\theta_1,\ldots,\theta_p)$ under $L_c$ 
if and only if $\kappa(Z)$ is admissible for $\gamma$ 
under $L(\gamma,\delta) = (\delta-\gamma)^2/\gamma$.
\end{theorem}

\begin{proof} 
The risk function of \eqref{eq:symform} under $L_c$ can be written
\beq
R_c(\delta,\theta)
 = \E_\gamma \Bigl\{ \frac{(1+c)^2Z }{p-1+(1+c)Z}L(\gamma,\kappa(Z)) 
   + \frac{(p-1)(1+c)\gamma}
{p-1+(1+c)Z}  \Bigr\},
\label{eq:risk.in.risk}
\eeq
with $L(\gamma,\delta) = (\kappa-\gamma)^2/\gamma$;  
see Appendix \ref{app:A} for the derivation of this claim.
Introduce $u(z) = (1+c)^2z/\{p-1+(1+c)z\}$ 
and write $\bar{L}(\delta,\gamma,z) = u(z)L(\gamma,\delta)$. 
Let $\delta$ and $\delta'$ be of the form \eqref{eq:symform} 
with $\kappa(Z)$ and $\kappa'(Z)$ respectively. If 
$\E_{\gamma} \big\{\bar{L}(\kappa'(Z),\gamma,Z)\big\} 
   -  \E_{\gamma}\big\{ \bar{L}(\kappa(Z),\gamma,Z)\big\} \leq 0$, 
with strict inequality for some $\gamma$, then 
$R_c(\delta,\theta) - R_c(\delta',\theta) \leq 0$ 
with strict inequality for some $\theta$. This shows that 
if $\kappa(Z)$ is inadmissible under $L^*$, then $\delta$ 
is inadmissible under $L_c$. 

The contrapositive statement is more enlightening: 
If $\delta$ is admissible under $L_c$ 
then $\kappa(z)$ is admissible under $\bar{L}$. 
Conversely, assume that $\kappa(Z)$ is admissible for $\gamma$ 
under $L$, let $\{\nu_n(\dd\gamma)\}_{n\geq1}$ be a sequence 
of prior measures satisfying the conditions of 
\citet[Theorem 2.6]{BrownFarrel1988}, and let 
$\rho_n(A)=\Gamma(p)\int_A\nu_n(\dd\gamma)\prod_{i=1}^p \dd\pi_i$ 
be the prior measure over $(0,\infty)\times S$, 
where $S$ is the $(p-1)$-dimensional simplex. Then, 
since $0 < u(z) < 1+c$ for all $z\geq 1$, 
\beq
\text{BR}_c(\delta,\rho_n) - \text{BR}_c(\delta_n,\rho_n) \leq (1+c) 
   \int \E_{\gamma} \{ L(\gamma,\kappa(Z)) 
   -  L(\gamma,\kappa_n(Z)) \}\,\dd\nu_n(\gamma).
\notag
\eeq
The right hand side is non-negative for all $n$ since 
$\delta_n$ is the Bayes solution, and it also 
tends to zero by Theorem 2.6 in \citet{BrownFarrel1988} 
because $\kappa(z)$ is admissible. This implies 
admissibility of $\delta$ under $L_c$. 
\end{proof}

The immediate corollary to Theorem \ref{th:admissibility} 
is that the estimator $\delta^c$ of \eqref{eq:delta1} 
is admissible under $L_c$, because $\kappa(Z) = Z$ 
is admissible under the weighted squared error loss 
function $L(\gamma,\delta) = (\delta - \gamma)^2/\gamma$.
As a special case, the Clevenson--Zidek estimator 
(\ref{eq:czestimator}) is admissible under 
the $\sumip(\delta_i-\theta_i)^2/\theta_i$ loss function. 

\section{Smoothing towards the mean}
\label{section:smoothtomean}

In the following we consider different strategies for 
smoothing the maximum likelihood estimator towards the 
mean of the observations. Pushing the maximum likelihood 
estimator towards a data-based point should in many cases 
yield more reduction in risk than pushing $\delta_0$ towards 
the origin, particularly when the $\theta_i$ are not small  
and not too spread out. This is clearly visible 
in Figure \ref{figure:risk1} where the improvement in 
risk of $\delta^c$ compared to $\delta_0$ 
becomes smaller as $\gamma = \sumip \theta_i$ grows; 
the improvement in risk is, in other words, 
biggest near the `point of attraction'.

\citet*{GhoshHwangTsui1983} considered a modification 
of $\delta_{\text{CZ}}$ that shrinks the maximum likelihood 
estimator towards the minimum of the observations, 
and were able to prove uniform dominance under 
the weighted loss $L_1^*$ of (\ref{eq:twolossfunctions}),  
for their modified estimator. 
\citet{Albert1981} took the Bayes estimator under 
$L_0(\theta,\delta) = \sumip (\delta_i - \theta_i)^2$ 
as his point of departure, and developed an estimator 
that pushes the observations towards the mean $\bar y$
and performs better than the usual estimator
in large parts of the parameter space.  

A complication when working with the $L_c$ loss function
of (\ref{eq:closs}) is that we compete with the maximum likelihood 
estimator on two turfs, so to speak, namely under 
$\sumip(\delta_i-\theta_i)^2/\theta_i$ {\it and} 
under $\big(\sumip\delta_i - \gamma \big)^2/\gamma$.   
One reason for choosing the mean as our `point of attraction' 
is that several of the estimators we construct 
preserve the mean, that is, $\sumip \delta_i(Y) = Z$. 
In this way such new estimators will always `match' 
the risk performance of $\delta_0$ when it comes 
estimating $\gamma$, and the penalty parameter 
$c$ becomes immaterial.         

\subsection{A restricted minimax estimator.}

Consider estimators of the form
\beq
\delta_i(Y) = Y_i - g(Z)(Y_i - \bar Y),
\label{eq:mean.shrink1}
\eeq
where $\bar Y = Z/p$. Notice that 
$\sumip \delta_i = Z = \sumip \delta_{0,i}$, which means that 
in calculations of the difference in risk between estimators 
of the form \eqref{eq:mean.shrink1} and $\delta_ 0 = Y$, 
the $c$-term in the $L_c$ loss function disappears.
The value of $g(0)$ is immaterial, since $\delta$ is then 
equal to zero, and we may take $g(0)=0$ for convenience. 

Using that $Y$ given $Z=z$ is multinomial with cell 
probabilities $\pi_i$, for each $z\ge1$, 
the risk difference $R_c(\delta,\theta) - R_c(\delta_0,\theta)$ 
can be expressed as 
\beqn
d(\theta)  
& =&\E_\theta\sumip \frac{1}{\theta_i} 
   \left\{ g(Z)^2(Y_i - Z/p)^2 - 2g(Z)(Y_i - Z/p)(Y_i - \theta_i) \right\}\\
& =& \E_\theta \sumip \left[   g(Z)^2\left\{  Z\frac{1-\pi_i}{\gamma} 
   + \pi_i \frac{Z^2}{\gamma} - 2\frac{Z^2}{p\gamma} 
   + \frac{Z^2}{p^2\theta_i}       \right\}   \right. \\   
& &\hspace{2cm} \left.  - 2g(Z)\left\{ Z\frac{1-\pi_i}{\gamma} 
   + \pi_i \frac{Z^2}{\gamma} - Z\pi_i - \frac{Z^2}{p\gamma} 
   + \frac{Z}{p}    \right\}      \right]\\
& =& \E_\theta \left[ g(Z)^2 \left\{ Z\frac{p-1-Z}{\gamma}  
   + \frac{Z^2}{p^2\gamma} \sumip\pi_i^{-1} \right\} 
 - 2g(Z)  Z\frac{p-1}{\gamma}    \right]\\
& =& \E_\theta \left[g(Z)^2 \left\{ Z\frac{p-1-Z}{\gamma}  
   + \frac{Z^2}{p\gamma} B(\pi) \right\} 
   - 2g(Z)Z\frac{p-1}{\gamma}  \right]\\
& =& \E_\theta \Bigl( g(Z+1)^2\left[p-1 + (Z+1)\{B(\pi)/p -1\}\right]
   - 2(p-1)g(Z+1) \Bigr), 
\eeqn
writing 
\beq
\label{eq:Bofpi}
B(\pi) = p^{-1} \sumip \pi_i^{-1}
\eeq  
for the mean of the inverse proportions. It is not possible 
to find a function $g(\cdot)$ such that the estimator in 
\eqref{eq:mean.shrink1} dominates $\delta_0 = Y$ over 
the entire parameter space $\Theta = (0,\infty)^p$. 
This is the `tyranny of the boundary' phenomenon; 
as a single $\theta_i \to 0$ the sum $B(\pi) \to \infty$ 
and the risk blows up for non-null choices of $g$. 

Estimators can be found, however, that dominate 
$\delta_0$ over large proportions of the parameter space. 
Let $\Theta(b_0)\subset \Theta$ be the subset of the 
parameter space where $B(\pi) \leq pb_0$. The minimum 
value of $B(\pi)$ is $p$, so $b_0 > 1$.  
In some situations one might have prior grounds for 
believing that the $\theta_i$ are somewhat similar, 
that is, not too far from the mean $\bar\theta$. One 
might for example have that each 
$\theta_i \geq \eps \bar{\theta} = \eps\gamma/p$, 
or equivalently $\pi_i \geq \eps/p$, for 
some small $\eps > 0$. This implies that 
$B(\pi) \leq p/\eps$, so $b_0=1/\eps$ may be used, 
so the risk difference can be bounded:  
\beq
d(\theta)\leq \E_\theta\,g(Z+1) 
   \left\{g(Z+1)[p-1+(Z+1)(b_0 -1)]  - 2(p-1)  \right\}.
\label{eq:b0ineq}
\eeq
Based on this upper bound we derive the estimator with 
\beq
\delta^*_i(Y) = Y_i - \frac{p-1}{p-1+(b_0 - 1)Z} (Y_i - \bar Y)
   \quad {\rm for\ }i=1,\ldots,p. 
\label{eq:mean.shrink2}
\eeq
This estimator succeeds in having 
\beqn
R_c(\delta^*,\theta) 
&\leq& 
   R_c(\delta_0,\theta) - \E_\theta \frac{(p-1)^2}{p-1+(b_0 - 1)(Z+1)} \\
& < & p+c - (p-1)^2/\{p-1 + (b_0 -1)(\gamma + 1)\}
\eeqn 
for all $\theta \in \Theta(b_0)$, 
with the last inequality following from Jensen's. 

\subsection{A restricted Bayes estimator.}

The estimator in \eqref{eq:mean.shrink2} was derived 
with the aim of risk function dominance in a given 
large region $\Theta(b_0)$ of the parameter space. 
We may also derive the restricted Bayes solution, 
that is, the best estimator among those of the 
form \eqref{eq:mean.shrink1}, under a prior 
of the type discussed in Section \ref{section:Bayes}. 
Since $\gamma$ and $\pi$ are independent, 
the Bayes risk of such an estimator is 
\beqn
\text{BR}_c(\delta) 
   &=& p + c - \E^q\,\E_{\gamma}\big\{2(p-1)g(Z+1) \\
& & \qquad  
   - g(Z+1)^2[p-1 + \{\E^r\,B(\pi)/p - 1\}(Z+1)] \big\}, 
\eeqn
where $\E^q\,(\cdot)$ and $\E^r\,(\cdot)$ are the expectations 
of $\gamma$ and $\pi$ with respect to their prior 
distributions, cf.~Lemma \ref{lemma:lemma1}. 
If we now let $\E^r\,B(\pi)=pb_0$, 
this reproduces the estimator in \eqref{eq:mean.shrink2}, 
but with a differently interpreted $b_0$. The risk 
function is 
\beq
R_c(\delta^B,\theta) = p + c
-(p-1)\,\E_{\gamma}\,g(Z+1) 
   \Bigl[2 - \frac{p-1+\{B(\pi)/p - 1\}(Z+1)}{p-1+(b_0-1)(Z+1)} \Bigr],
\notag
\eeq    
with $g(z) = (p-1)/\{p-1+(b_0 -1)z\}$. Consider again the prior 
construction of Section \ref{section:Bayes}. If $(\pi_1,\ldots,\pi_p)$
comes from a Dirichlet distribution with parameters 
$(\alpha,\ldots,\alpha)$, then $b_0 = (\alpha - 1/p)/(\alpha-1)$,  
provided that $\alpha > 1$. Recall that the inequality in 
\ref{eq:b0ineq} only holds if $b_0 >1$, which means that the 
subspace $\Theta(b_0)$ is empty if our prior knowledge 
dictates $0 < \alpha \leq 1$. 
On the other side of the spectrum, 
as $\alpha \to \infty$ this estimator would assign the estimate 
$1/p$ to each of the proportions, and since $b_0$ goes to one 
as $\alpha$ goes to infinity, the estimator \eqref{eq:mean.shrink2} 
tends to $\bar y = z/p$. In other words, having 
smaller values of $\alpha > 1$ expands the space 
$\Theta(b_0)$ but results in less gain in risk.  


The (\ref{eq:mean.shrink2}) estimator may also use 
a data-based value for $b_0$. With $\pi$ and $\gamma$ 
independent, the marginal distribution of the data is   
\beqn
\int_S\int_0^{\infty} f(y\midd\pi,\gamma) r(\pi) 
   q(\gamma)\,\dd\gamma\,\dd\pi_1 \cdots \dd\pi_{p-1} 
\propto \frac{\Gamma(p\alpha)}{\Gamma(\alpha)^p} 
   \frac{\prod_{i=1}^p\Gamma(\alpha + y_i)}{\Gamma(p\alpha + z)}.
\eeqn 
This likelihood can be maximised to obtain an estimate 
$\hatt\alpha$ which is plugged into $b_0=(\alpha - 1/p)/(\alpha-1)$, 
again provided that $\hatt\alpha > 1$.          

\subsection{More careful smoothing towards the mean.}

We now consider the construction
\beqn
\hatt\theta_i=Y_i-g(Z)(Y_i-\bar Y)h(Y), 
\eeqn
where $h(y)$ is a function such that if one or more 
$y_i=0$, then $h(y)=h(y_1,\ldots,y_p)=0$. The intention
is that of more careful smoothing towards the mean 
than with (\ref{eq:mean.shrink1}), to achieve risk improvement
in potentially larger parts of the parameter space. 
Note that $\sumip\hatt\theta_i=\sumip y_i$, so any risk
difference $R_c(\hatt\theta,\theta)-R_c(Y,\theta)$
does not depend on the $c$ term of the loss function. 
Also, the value of $g(z)$ at $z=0$ is immaterial,
so we may take $g(0)=0$, for convenience. 

To work with the risk functions, we start from 
\beqn
(\hatt\theta_i-\theta_i)^2-(y_i-\theta_i)^2
   =g(z)^2(y_i-\bar y)^2h(y)^2
   -2g(z)(y_i-\bar y)h(y)(y_i-\theta_i). 
\eeqn
The risk difference $R_c(\hatt\theta,\theta)-R_c(Y,\theta)$ 
may hence be expressed as 
\beqn
\E_\theta\sumip \{{g(Z)^2(Y_i-\bar Y)^2h(Y)^2
   -2g(Z)(Y_i-\bar Y)h(Y)(Y_i-\theta_i)}\}/\theta_i 
   =\E_\theta\,D(Y), 
\eeqn 
with 
\beqn
D(y)
&=&\sumip\Bigl[{g(z+1)^2\{y_i+1-(z+1)/p)\}^2 h(y+e_i)^2\over y_i+1} \\
& &\qquad 
   -2g(z+1)\{y_i+1-(z+1)/p\}h(y+e_i)
   +2g(z)(y_i-\bar y)h(y) \Bigr] \\
&=&g(z+1)^2 \Bigl\{\sumip (y_i+1)h(y+e_i)^2 \\
& & \qquad -2{z+1\over p}\sumip h(y+e_i)^2
   +{(z+1)^2\over p^2}\sumip {h(y+e_i)^2\over y_i+1}\Bigr\} \\
& &\qquad 
   -2g(z+1)\Bigl\{\sumip(y_i+1)h(y+e_i)-{z+1\over p}\sumip h(y+e_i)\Bigr\}.
\eeqn 
The property that $y_i=0$ implies $h(y)=0$ is actively
used here; without such a constraint, more complicated
terms need to be added to the $\E_\theta\,D(Y)$ here. 

Several choices may be considered and worked with
for further fine-tuning, regarding the $h(y)$ function.
For the present report we limit attention to the special 
case of $h_0(y)=I\{{\rm each\ }y_i\ge1\}$. We need to study 
and bound the $D(y)$ function with this choice of $h(y)$. 
Note that $\sumip h_0(y+e_i)$ is $p$, if all $y_i\ge1$; 
is 1, if one and only one of $y_i=0$;  
and is 0, otherwise. Similarly, 
$\sumip y_ih_0(y+e_i)=z$, if all $y_i\ge1$; 
and 0, otherwise. Furthermore, 
\beqn
\sumip{h_0(y+e_i)\over y_i+1}
   \begin{cases}
   \le p/2 &\quad{\rm if\ all\ }y_i\ge1, \\
   =1      &\quad{\rm if\ only\ one\ }y_i=0, \\
   =0      &\quad{\rm otherwise}.
   \end{cases} 
\eeqn
(i) Assume first that all $y_i\ge1$. Then 
\beqn
D(y)
&=&g(z+1)^2 \Bigl\{z+p-2(z+1)+{(z+1)^2\over p^2}
   \sumip{1\over y_i+1}\Bigr\}
   -2g(z+1)(p-1) \\
&\le&g(z+1)^2\Bigl\{p-1-(z+1)+\half{(z+1)^2\over p}\Bigr\}
   -2g(z+1)(p-1). 
\eeqn 
The function $p-1-x+\half x^2/p$ for $x\ge1$
has its minimum value at position $x=p$,
with minimum value $\half p-1$, which is positive 
as long as $p\ge3$. 
(ii) When there is only one $y_i=0$, the rest $y_j\ge1$, 
we find
\beqn 
D(y)
&=&g(z+1)^2 \Bigl\{0 + 1 -2{z+1\over p}1
   +{(z+1)^2\over p^2}1\Bigr\}
   -2g(z+1)\Bigl\{0+1-{z+1\over p}0\Bigr\} \\
&=&g(z+1)^2\Bigl\{-2{z+1\over p}
   +{(z+1)^2\over p^2}\Bigr\}
   -2g(z+1). 
\eeqn 
(iii) Otherwise, which means that the number 
of $y_i=0$ is between 2 and $p-1$, 
we find that $D(y)=0$. 

Our best choice for $g(z)$, based on the upper risk bound
for the case of all $y_i\ge1$, is 
\beqn
g_0(z+1)={p-1\over p-1-(z+1)+\half(z+1)^2/p},
   \quad{\rm or}\quad
g_0(z)={p-1\over p-1-z + \half z^2/p}. 
\eeqn
The estimator 
\beq
\label{eq:coolnewguy}
\begin{array}{rcl}
\hatt\theta_i
&=&\displaystyle 
Y_i-g_0(Z)(Y_i-\bar Y)h_0(Y) \\ 
&=&\displaystyle 
Y_i-{p-1\over p-1-Z + \half Z^2/p}(Y_i-\bar Y)I\{{\rm all\ }Y_i\ge1\}, 
\end{array}
\eeq 
therefore, has risk function 
$R_c(\hatt\theta)=p+c+\E_\theta\,D_0(Y)$, where an exact expression
for $D_0(Y)$ is found via the above. 
We also know that 
\beqn
D_0(y)\le-{(p-1)^2\over p-1-(z+1)+\half (z+1)^2/p} 
   \quad {\rm if\ all\ }y_i\ge1\,;
\eeqn 
that 
\beqn
D_0(y)
&=&\Bigl\{{p-1\over p-1-(z+1)+\half(z+1)^2/p}\Bigr\}^2
   \Bigl\{-2{z+1\over p}+{(z+1)^2\over p^2}\Bigr\} \\
& &\qquad\qquad 
   -2{p-1\over p-1-(z+1)+\half(z+1)^2/p},  
\eeqn 
for the cases of $y$ where precisely one $y_i=0$,
the other $y_j\ge1$; and finally that $D_0(y)=0$
for those $y$ for which the number of $y_i=0$
is between 2 and $p-1$. 
As long as all $\theta_i$ are at least moderate, 
so that $h_0(Y)=1$ with high probability, there 
is clear risk improvement on the minimax risk $p+c$. 
The corner cases, however, where one $\theta_i$
is small and the others not, are the `bad cases'
for the (\ref{eq:coolnewguy}) estimator, 
where the risk might become larger than $p+c$. 
Since $D_0(y)$ flattens to zero when $z=\sumip y_i$
increases, the risk converges to $p+c$ 
for all $\theta=\gamma(\pi_1,\ldots,\pi_p)$ where 
$\gamma$ tends to infinity.

\section{Bayes and empirical Bayes with more structure}
\label{section:moreempiricalbayes}

In the previous sections we have constructed and analysed
estimators essentially symmetric in the observations.
Sometimes some structure is anticipated in the parameters, 
however, as with setting up regressions or log-linear models 
for Poisson tables, see e.g.~\citet{Agresti2019}, 
or with classes of priors. The present section briefly 
complements our earlier efforts by examining risk function 
consequences for estimators that favour asymmetric 
structures. 

\subsection{Risk functions with Gamma priors.}

A natural class of priors takes independent 
Gamma priors, with parameters $(\alpha_i,\beta)$,
for the $\theta_i$. As was seen in Section \ref{section:closs}, 
the Bayes estimator then takes the form 
\beq
\label{eq:nils31} 
\delta_i^B=h_c(z){\alpha_i+y_i-1\over \beta+1}
   \quad {\rm with} \quad 
   h_c(z)={(1+c)(p\bar\alpha+z-1)\over (1+c)(p\bar\alpha+z-1)-c(p-1)},
\eeq
writing $\bar\alpha$ for $(1/p)\sumip\alpha_i$; 
also, for cases where $\alpha_i<1$ and $y_i=0$, 
the estimator is zero. 

The present task is to study the consequent risk functions, 
for such estimators, outside the special and simplest case 
where each $\alpha_i=1$. The point will be that 
estimators then lose the minimax property, with 
the risk exceeding the $p+c$ benchmark level 
when one or more of the $\theta_i$ come close to zero,
but that the risk otherwise can be lower than $p+c$
in big and reasonable parts of the parameter space. 
For simplicity of presentation we restrict attention
here to the case of $c=0$. Similar results 
and insights may be reached for the general loss 
function $L_c$, using more laborious methods we develop 
and exploit for somewhat different purposes in 
Section \ref{subsection:riskagain}. 
For the ensuing estimator 
$(\alpha_i+Y_i-1)/(\beta+1)$, some calculations yield 
\beqn
R(\delta^B,\theta)
&=&\sumip {1\over \theta_i}\Bigl[\Bigl\{{\alpha_i+\theta_i-1-(\beta+1)\theta_i
   \over \beta+1}\Bigr\}^2+{\theta_i\over (\beta+1)^2}\Bigr] \\
&=&{1\over (\beta+1)^2}
   \Bigl\{\sumip {(\alpha_i-1-\beta\theta_i)^2\over\theta_i}+p\Bigr\}. 
\eeqn 
This is smaller than or equal to $R(Y,\theta)=p$ when 
\beqn
A(\theta)={1\over p}\sumip {(\alpha_i-1-\beta\theta_i)^2\over \beta\theta_i}
   \le {(\beta+1)^2-1\over \beta}=2+\beta. 
\eeqn 
This defines a fairly large parameter region, 
$\{\theta\colon A(\theta)\le 2+\beta\}$, 
with the $1/\theta_i$ not being too far away from 
the prior mean values $\beta/(\alpha_i-1)$, 
and where using the Bayes estimator hence is better 
than with the $\delta_0=Y$ method. Under the prior itself, 
the random $A(\theta)$ has mean 
\beqn 
{1\over p}
   \sumip \Bigl\{{(\alpha_i-1)^2\over \beta}{\beta\over \alpha_i-1}
   -2(\alpha_i-1)+\beta{\alpha_i\over \beta}\Bigr\}=1,
\eeqn 
and variance of order $O(1/p)$, showing that 
only rather unlikely values of $\theta$ will have risk 
above the benchmark $p$. 

\subsection{Risks for a class of empirical Bayes estimators.}
\label{subsection:riskagain} 

With the independent Gamma priors used above 
we next note that the marginal distribution of $y_i$ becomes 
\beqn
g(y_i\midd\alpha_i,\beta)
   &=&\int_0^\infty {\beta^{\alpha_i}\over \Gamma(\alpha_i)}
   \theta_i^{\alpha_i-1}\exp(-\beta\theta_i)
   {1\over y_i!}\exp(-\theta_i)\theta_i^{y_i}\,\dd\theta_i \\
&=&{\beta^{\alpha_i}\over \Gamma(\alpha_i)}{1\over y_i!}
   {\Gamma(\alpha_i+y_i)\over (\beta+1)^{\alpha_i+y_i}}, 
\eeqn
for $y_i=0,1,2,\ldots$. In the setup where the $\alpha_i$
are taken known but $\beta$ an unknown parameter, 
we see that $Z=\sumip Y_i$ is sufficient. Since 
$Z$ given the parameters is Poisson with mean 
$\gamma=\sumip\theta_i$, and $\gamma$ is Gamma $(p\bar\alpha,\beta)$, 
writing $\bar\alpha=(1/p)\sumip\alpha_i$, its distribution 
may be written 
\beqn
q(z\midd\beta)={\beta^{p\bar\alpha}\over \Gamma(p\bar\alpha)}
   {1\over z!}{\Gamma(p\bar\alpha+z)\over (\beta+1)^{p\bar\alpha+z}}
   \quad {\rm for\ }z=0,1,2,\ldots. 
\eeqn
From this we can derive 
\beq
\label{eq:nils15}
\E\,{Z\over p\bar\alpha-1+Z}
   =\sum_{z=1}^\infty {1\over (z-1)!}{\beta^{p\bar\alpha}\over \Gamma(p\bar\alpha)}
   {\Gamma(p\bar\alpha+z-1)\over (\beta+1)^{p\bar\alpha+z-1}}{1\over \beta+1}
   ={1\over \beta+1},  
\eeq
provided only that $p\ge2$. Hence $Z/(p\bar\alpha-1+Z)$
can be used as an estimator for the quantity $1/(\beta+1)$.
With pre-specified $\alpha_i$, then, a natural
empirical Bayes version of estimator (\ref{eq:nils31}) emerges: 
\beq
\label{eq:nils32} 
\begin{array}{rcl}
\hatt\theta_i
&=&\displaystyle
{(1+c)(p\bar\alpha-Z-1)\over (1+c)(p\bar\alpha+Z-1)-c(p-1)}
   {Z\over p\bar\alpha-1+z}(\alpha_i+Y_i-1) \\
&=&\displaystyle 
h_c(Z)(\alpha_i+Y_i-1), 
\end{array}
\eeq
with 
\beqn
h_c(Z)={(1+c)Z\over (1+c)(p\bar\alpha+Z-1)-c(p-1)}. 
\eeqn 
In particular, under $L_0$ loss, with $c=0$, the 
natural empirical Bayes estimator is 
\beq
\label{eq:nils11}
\hatt\theta_i=h_0(Z)(\alpha_i+Y_i-1)
   ={Z\over p\bar\alpha+Z-1}(\alpha_i+Y_i-1), 
\eeq 
generalising the earlier symmetric case with all $\alpha_i=1$,
which yields the already studied minimax and admissible 
estimator (\ref{eq:emilnilsestimator}). 

Expressions for the risk function $R_c(\hatt\theta,\theta)$
can now be worked out, using the fact that 
$y_i\midd z$ is binomial $(z,\pi_i)$, with $\pi_i=\theta_i/\gamma$. 
Consider the general class of estimators 
\beq 
\label{eq:nils12}
\hatt\theta_i=h(Z)(\alpha_i+Y_i-1)
   \quad {\rm for\ }i=1,\ldots,p.
\eeq 
The $h(z)$ functions we are encountering all have $h(0)=0$, 
and will in fact have the form $q(z)z$, for suitable $q(z)$. 
Also, they will be nondecreasing with $h(z)\arr 1$ as $z$ increases. 
The task now is to develop formulae for their risk 
functions, through suitable representations of the form 
\beqn
R(\hatt\theta,\theta)=\E_\theta\,H(Z,\theta)
   =\sum_{z=1}^\infty H(z,\theta)\exp(-\gamma)\gamma^z/z!, 
\eeqn
and then showing, for relevant choices of $h(z)$, 
that this is less than $p+c$ for large and 
relevant parameter regions. 

We start from 
\beqn
\E_\theta\,{1\over\theta_i}\{h(z)(\alpha_i+Y_i-1)-\theta_i\}^2\midd z
&=&{1\over \gamma\pi_i}
   [\{h(z)(\alpha_i+z\pi_i-1)-\gamma\pi_i\}^2 \\
& &\qquad\qquad +\,h(z)^2z\pi_i(1-\pi_i)] \\
&=&{1\over \gamma\pi_i} \bigl\{[h(z)(\alpha_i-1)+\{h(z)z-\gamma\}\pi_i]^2 \\
& &\qquad\qquad +\,h(z)^2z\pi_i(1-\pi_i)\bigr\}.
\eeqn 
For the case of $c=0$ this leads to 
\beqn
H(z,\theta)
&=&h(z)^2\sumip{(\alpha_i-1)^2\over \gamma\pi_i}
   +\{h(z)z-\gamma\}^2{1\over\gamma} \\
& &\qquad\qquad
   +\,2\{h(z)z-\gamma\}h(z)(p\bar\alpha-p){1\over \gamma}
   +h(z)^2z(p-1){1\over \gamma}. 
\eeqn 
For the `$c$ term' part of the risk, 
note that $\hatt\theta_i=h(z)(p\bar\alpha+z-p)$, 
leading to 
\beqn
\E_\gamma\,{\{h(Z)(p\bar\alpha+Z-p)-\gamma\}^2\over \gamma} 
=\E_\gamma\,{ \{h(Z)Z-\gamma + h(Z)(p\bar\alpha-p)\}^2\over \gamma}. 
\eeqn 
This may be exploited further using the identity 
$\E_\gamma\,r(Z)/\gamma=\E_\gamma\,r(Z+1)/(Z+1)$  
for functions $r(z)$ with $r(0)=0$. 

For brevity we limit attention here to the case 
of $c=0$; extensions can be worked out using 
the same methods. We use the identity pointed to for 
\beqn
\{h(z)z-\gamma\}^2/\gamma 
   =h(z)^2z^2/\gamma-2h(z)z+\gamma, 
\eeqn 
and find $R(\hatt\theta,\theta)=\E_\gamma\,Q(Z,\theta)$, with 
\beqn
Q(z,\theta)&=&h(z)^2\sumip{(\alpha_i-1)^2\over \gamma\pi_i}
   +h(z+1)^2(z+1)-2h(z)z+\gamma \\ 
& & \qquad
   +\,2{ \{h(z+1)(z+1)-\gamma\}h(z+1)\over z+1}(p\bar\alpha-p)
   +h(z+1)^2(p-1). 
\eeqn 
The risk function may hence be expressed as  
\beq
R(\hatt\theta,\theta)
   =pm_1(\gamma)\,C(\pi)/\gamma
   +R_1(\gamma)+2(p\bar\alpha-p)R_2(\gamma)+(p-1)m_2(\gamma), 
\label{eq:nils14}
\eeq
in which $C(\pi)=(1/p)\sumip (\alpha_i-1)^2/\pi_i$, and 
\beqn
m_1(\gamma)&=&\E_\gamma\,h(Z)^2, \\
m_2(\gamma)&=&\E_\gamma\,h(Z+1)^2, \\
R_1(\gamma)&=&\E_\gamma\,\{h(Z+1)^2(Z+1)-2h(Z)Z+\gamma\}, \\
R_2(\gamma)&=&\E_\gamma\,{\{h(Z+1)(Z+1)-\gamma\}h(Z+1)\over Z+1}.
\eeqn  
Here $m_1(\gamma)$ and $m_2(\gamma)$ are inside $(0,1)$, 
and increase to 1; the $R_1(\gamma)$ and $R_2(\gamma)$ 
are bounded and converge to respectively one and zero
as $\gamma$ increases. For larger $\gamma$, therefore, 
the risk goes to the minimax risk $p$. 
The risk function (\ref{eq:nils14}) may exceed 
the minimax threshold level $p$ if one or more
of the $\pi_i=\theta_i/\gamma$ are small, but even
for small $\pi_i$ the risk decreases with increasing
$\gamma$. Otherwise the situation is that the risk 
may become significantly smaller than $p$ in 
parts of the parameter space not disagreeing much
from what is judged likely under the prior, and 
that it can be smaller than $p$ also in other 
larger parameter regions. An upper bound is 
\beqn
R(\hatt\theta,\theta)
   \le pm_1(\gamma)\{\max_{i\le p}|\alpha_i-1|\}^2B(\pi)/\gamma
   +R_1(\gamma)+2(p\bar\alpha-p)R_2(\gamma)+(p-1)m_2(\gamma), 
\eeqn 
with $B(\pi)$ as in (\ref{eq:Bofpi}). 
This may in particular be investigated further,
with the choice $h_0(z)$, corresponding to 
the estimator (\ref{eq:nils11}). The risk function
is bounded; converges to the minimax value $p$
when $\gamma$ increases, regardless of proportions 
$(\pi_1,\ldots,\pi_p)$; may offer substantial improvement
for sizeable portions of the parameter space; 
and its maximum value is often not much bigger 
than $p$.  

\begin{figure}[ht]
\centering
\includegraphics[scale=0.55,angle=0]{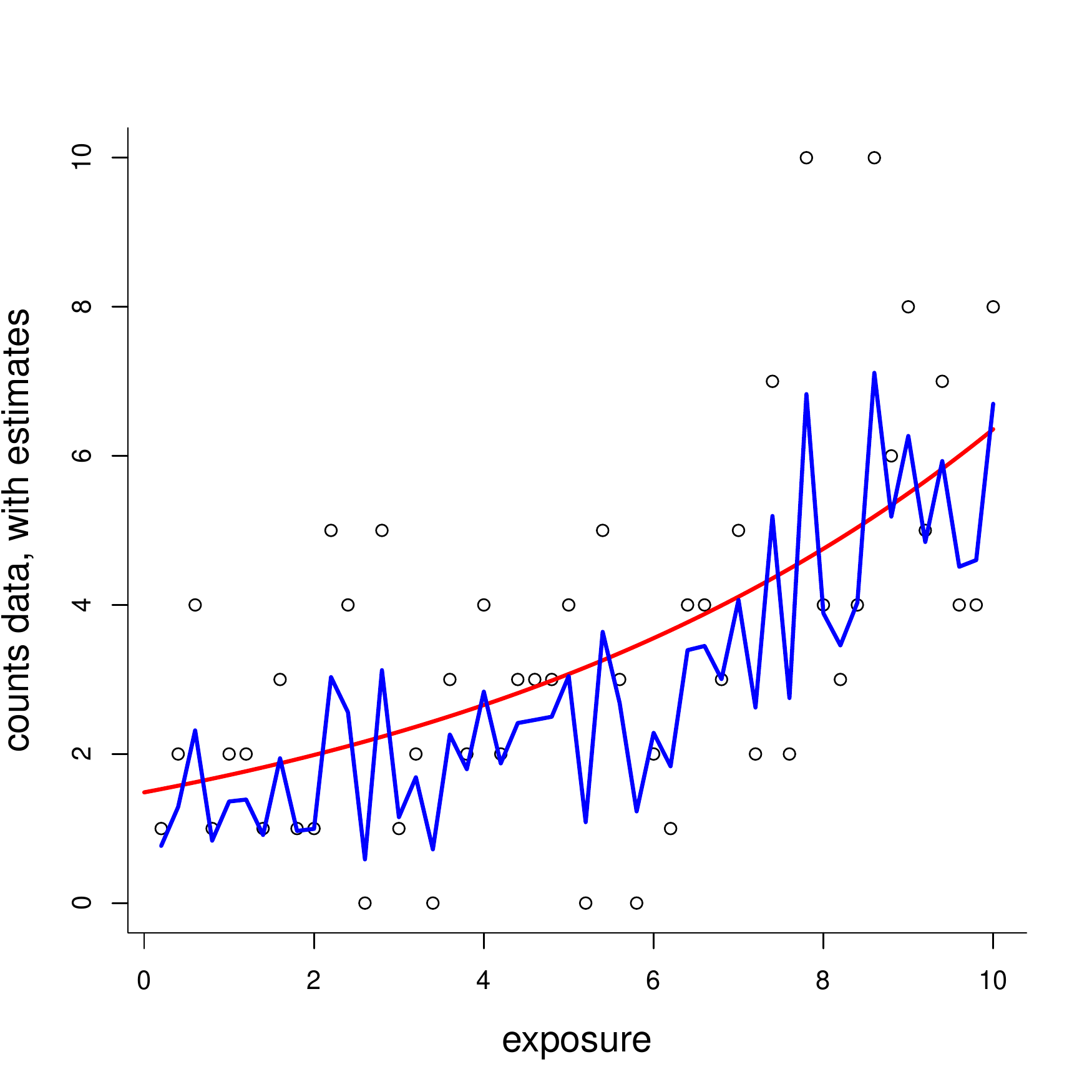}
\caption{Estimating $p=50$ Poisson parameters, 
in a simulated regression setting with $(x_i,y_i)$ data,  
with the empirical Bayes method (\ref{eq:nils11}),  
shrinking the raw estimates $y_i$ towards
prior means $\alpha_i/\beta$, with $\beta$ estimated 
from the data.} 
\label{figure:figure2}
\end{figure}

An illustration of the empirical Bayes strategy (\ref{eq:nils11})
is provided in Figure \ref{figure:figure2}, in 
a situation with simulated regression data $(x_i,y_i)$.
The prior takes the $\theta_i$ to stem from Gamma 
$(\alpha_i,\beta)$, with $\alpha_i=\exp(\gamma_0+\gamma_1 x_i)$,
for suitable prior guess values $(\gamma_0,\gamma_1)$,
and then estimates $\beta$ from data, 
as per (\ref{eq:nils15}). 
Similar Bayes and empirical Bayes methods can be developed
for priors of the type $\theta_i\sim\Gam(d\theta_{0,i},d)$,
with either hyperpriors on the prior parameters 
$\theta_{0,i}$ and $d$, or involving estimators for 
these from the data. 

\section{Estimation with a weighted loss function}
\label{section:weightedloss}

Above we have worked with our loss function $L_c$, 
a natural extension of the Clevenson--Zidek loss function
$L_1^*(\theta,\delta)=\sumip(\delta_i-\theta_i)^2/\theta_i$,
to account for not shrinking the mean too much. 
Another useful extension is to allow for weighting, with 
\beqn
L_w(\theta,\delta)=\sumip w_i(\delta_i-\theta_i)^2/\theta_i, 
\eeqn 
where $w_1,\ldots,w_p$ are fixed, positive, and context 
driven, reflecting relative importance. This is 
e.g.~important when the $y_i$ result from different
levels of exposure, as with $y_i$ stemming from a 
Poisson with parameter $w_i\theta_i$. The $\delta_0=Y$ 
estimator is again the natural benchmark; it has 
constant risk $w_0=\sumip w_i$, and it is minimax. 
To prove this second claim, one may work with the 
prior where the $\theta_i$ are independent and 
Gamma distributed, with parameters $(\alpha,\beta)$,
and with $\alpha\ge1$ to avoid a certain technical issue. 
Some work shows that the Bayes estimator becomes 
$\hatt\theta_i=1/\E\,(\theta_i\midd y)=(\alpha+y_i-1)/(\beta+1)$,
with associated minimum Bayes risk as simple as $w_0/(\beta+1)$.
Letting $\beta\arr0$ we have convergence to the 
constant risk of $\delta_0$. 

This estimator may be uniformly improved upon, however,
as we now demonstrate, yielding another generalisation
of the Clevenson--Zidek estimator. Consider 
estimators of the form 
\beq
\label{eq:newestimators}
\hatt\theta_i=\{1-\phi(V)\}Y_i 
   \quad {\rm for\ }i=1,\ldots,p, 
   \quad {\rm where\ }V=v(Y)=\sumip w_iY_i. 
\eeq
Using the identity (\ref{eq:useful.identity}) we may 
express the risk difference $r_w$ between $\hatt\theta$ and $\delta_0$, 
i.e.~$\E_\theta\,\{L_w(\theta,\hatt\theta)-L_w(\theta,\delta_0)\}$, as
\beqn
r_w&=&\E_\theta\sumip {w_i\over \theta_i}
   \{\phi(V)^2Y_i^2-2\phi(V)Y_i(Y_i-\theta_i)\} \\
&=&\E_\theta\sumip w_i\bigl[\{\phi(v+w_i)^2(Y_i+1) 
   - 2\phi(v+w_i)(Y_i+1)+2\phi(V)Y_i\bigr]
   =\E_\theta\,D_w(Y), 
\eeqn 
say, where 
\beqn
D_w(y)=\sumip w_i\{\phi(v+w_i)^2-2\phi(v+w_i)\}(y_i+1)+2\phi(v)v. 
\eeqn
As we have argued on previous occasions in our paper, 
if we succeed in finding a function $\phi(v)$
such that $D_w(y)\le0$ for all $y$, with strict 
inequality for at least one $y$, we have demonstrated
inadmissibility of $\delta_0=Y$. 

To work with this we set up some mild requirements regarding
the $w_i$ weights. We take all $w_i$ to be inside some 
$[a,b]$ interval, situated inside $(0,1)$, 
and also stipulate that $w_0=\sumip w_i>1$. 
For estimators of the form (\ref{eq:newestimators}), consider 
\beqn 
\phi(v(y))={\psi(v(y))\over w_0-1+v(y)},
\eeqn 
with (i) $\psi(v)$ nondecreasing, 
(ii) with $0<\psi(v)<2(w_0-1)$ for all $v=v(y)$, and 
(iii) such that $v\phi(v)$ is increasing. 
We then find 
\beqn
D_w(y)
&=&\sumip w_i 
    {\psi(v+w_i)\over w_0-1+v+w_i}
    \Bigl\{{\psi(v+w_i)\over w_0-1+v+w_i}-2\Bigr\}(y_i+1)+2\phi(v)v \\
&\le&\sumip w_i 
    {\psi(v+a)\over w_0-1+v+b}
    \Bigl\{{\psi(v+w_i)\over w_0-1+v+w_i}-2\Bigr\}(y_i+1)+2\phi(v)v \\ 
&\le&\sumip w_i 
    {\psi(v+a)\over w_0-1+v+b}
    \Bigl\{{\psi(v+b)\over w_0-1+v+a}-2\Bigr\}(y_i+1)+2\phi(v)v \\
&=&{\psi(v+a)\over w_0-1+v+b}
   \Bigl\{{\psi(v+b)\over w_0-1+v+a}-2\Bigr\}(w_0+v)+2\phi(v)v \\
&\le&\psi(v+a)\Bigl\{{\psi(v+b)\over w_0-1+v+a}-2\Bigr\}+2\phi(v)v \\
&=&\psi(v+a)\{\psi(v+b)-2(w_0-1)\}-2\{\phi(v+a)(v+a)-\phi(v)v\}, 
 \eeqn
which is demonstrably negative. Our preferred generalisation 
of the Clevenson--Zidek estimator, 
to the present case of weighted loss, becomes
\beq
\label{eq:newguy}
\hatt\theta_i=\Bigl\{1-{w_0-1\over w_0-1+v(Y)}\Bigr\}Y_i
   \quad {\rm for\ }i=1,\ldots,p. 
\eeq 
The special case of equal weights $w_i=1$ leads back
to the Clevenson--Zidek estimator (\ref{eq:czestimator}). 

\def\sumibig{{\sum_{i=1}^\infty}}

Remarkably, the apparatus above allows extension to the case
of infinitely many Poisson parameters. Suppose $Y_1,Y_2,\ldots$
are independent Poisson counts with means $\theta_1,\theta_2,\ldots$,
and that loss incurred by estimators $\delta_1,\delta_2,\ldots$
is taken to be 
$L(\theta,\delta)=\sumibig w_i(\delta_i-\theta_i)^2/\theta_i$.
Here the sequence of weights is such that $w_0=\sumibig w_i$
is finite, and the parameter space to be considered 
is $\Omega$, the set of sequences of $\theta_i$ for which 
$\sumibig w_i\theta_i$ is finite (including in particular
each bounded sequence). 

The benchmark procedure is again $\delta_0$, 
with components $\delta_{0,i}=y_i$. It has constant risk $w_0$, 
and our previous arguments may be extended to demonstrate 
that this procedure is minimax. Also, crucially, the estimator 
\beq
\label{eq:infinite}
\hatt\theta_i=\Bigl\{1-{w_0-1\over w_0-1+v(Y)}\Bigr\}Y_i
   \quad {\rm for\ }i=1,2,3,\ldots 
\eeq 
offers uniform risk improvement over $\delta_0$,
where now $v(y)=\sumibig w_iy_i$.  
This follows from arguments used to reach 
the corresponding statement for the finitely-many
procedure (\ref{eq:newguy}), in the light of 
two necessary remarks. The first is that the 
identity (\ref{eq:useful.identity}) continues to hold, 
for all $h(Y_1,Y_2,\ldots)$ with finite mean,
with the property that $h(y)=0$ as long as $y_i=0$.
The second is that the upper bound reached above
for $D_w(y)$, under the condition that $a\le\theta_i\le b$
for all $i$, applies here too, but now we need $a=0$,
since the $\sumibig w_i$ is finite. In other words,   
\beqn
D_w(y)\le\psi(v(y))\{\psi(v(y)+b)-2(w_0-1)\} 
\eeqn
remains correct. 

Above we stipulated a scaling of the importance weights
$w_i$ so that their sum $w_0=\sumibig w_i$ is above 1.
This is partly in order for the estimator (\ref{eq:infinite})
to be a natural generalisation of the Clevenson--Zidek 
estimator. Similar reasoning goes through for estimators 
$[\psi(v(Y))/\{w_0 - \eps + v(Y)\}]\,Y$, if we instead
stipulate $w_0>\eps$. 

\section{Multivariate models for count variables} 
\label{section:extra} 


The methods developed in Section \ref{section:Bayes}  
utilised certain constructions which also involve 
multivariate models for rates and for count variables, 
of interest in their own right. Models can be built
with both positive and negative correlations 
betwen rates $\theta_i$ and between count observations $Y_i$.
These modelling ideas also point to Bayesian nonparametrics,
cf.~our Remark B in Section \ref{section:concluding}. 

\subsection{Sum and proportions models.}

Suppose first in general terms that given 
$\theta=(\theta_1,\ldots,\theta_p)$, 
the observations $Y_1,\ldots,Y_p$ have independent 
Poisson distributions with these parameters, and 
that the $\theta$ has a background distribution,
which we for simplicity of presentation here take
to be symmetric with finite variances. 
Let us write 
$\E\,\theta_i=\theta_0$,
$\Var\,\theta_i=\sigma_0^2$, 
$\cov(\theta_i,\theta_j)=\rho\sigma_0^2$ for $i\not= j$. 
We then deduce 
\beq
\label{eq:extra11}
\E\,Y_i=\theta_0, \quad
\Var\,Y_i=\theta_0+\sigma_0^2, \quad
\cov(Y_i,Y_j)=\rho\sigma_0^2, \quad 
\corr(Y_i,Y_j)={\rho\sigma_0^2\over \theta_0+\sigma_0^2}.   
\eeq

A class of multivariate models for the $\theta_i$,
and by implication also for the $Y_i$,  
emerges from the construction of Lemma \ref{lemma:lemma1},
with a prior $q(\gamma)$ for $\gamma=\sumip\theta_i$
and a symmetric prior $\Dir(\alpha,\ldots,\alpha)$ 
for the proportions $\pi_i=\theta_i/\gamma$.
Write $\E\,\gamma=\gamma_0=p\theta_0$ and 
$\Var\,\gamma=p\tau_0^2$. The $\pi_i$ have means $1/p$, 
variances $(1/p)(1-1/p)/(p\alpha+1)$, and covariances
$-(1/p^2)/(p\alpha+1)$. From these facts we first find 
$\E\,\theta_i=\E\,(\gamma\pi_i)=\theta_0$ and then 
\beqn
\Var\,\theta_i
   &=&\E\,(\gamma\pi_i)^2-\theta_0^2
   =(\gamma_0^2+p\tau_0^2)
   \Bigl\{{1\over p^2}+{1\over p}\Bigl(1-{1\over p}\Bigr)
   {1\over p\alpha+1}\Bigr\}-\theta_0^2 
=\theta_0^2{p-1\over p\alpha+1}+\tau_0^2{\alpha+1\over p\alpha+1}. 
\eeqn 
Similarly, some calculations lead to 
$\cov(\theta_i,\theta_j)=(\tau_0^2\alpha-\theta_0^2)/(p\alpha+1)$, 
so that the correlation parameter $\rho$ of (\ref{eq:extra11})
may be expressed as 
\beqn
\rho={\alpha\tau_0^2-\theta_0^2\over (\alpha+1)\tau_0^2+(p-1)\theta_0^2}. 
\eeqn 
For the special case of $\gamma\sim\Gam(p\alpha,\beta)$, 
we have $\theta_0=\alpha/\beta$ and $\tau_0^2=\alpha/\beta^2$, 
the covariance is zero, and the formula for 
$\Var\,\theta_i$ gives $\alpha/\beta^2$. This is indeed the 
familiar case of independent $\theta_i\sim\Gam(\alpha,\beta)$.

For other models for $\gamma$, however, 
the construction above leads to useful multivariate models 
for $(\theta_1,\ldots,\theta_p)$ and for $(Y_1,\ldots,Y_p)$. 
In general terms, if $\gamma$ has density $q(\gamma)$, 
the distribution of the data vector 
can be written as 
\beq
\label{eq:extra13}
\begin{array}{rcl}
f(y_1,\ldots,y_p)
&=&\displaystyle
\E\,\prod_{i=1}^p \exp(-\theta_i)\theta_i^{y_i}/y_i! 
=\E\,\exp(-\gamma)\gamma^z\pi_1^{y_1}\cdots\pi_p^{y_p}/(y_1!\cdots y_p!) \\ 
&=&\displaystyle
 K(z) {\Gamma(p\alpha)\over \Gamma(\alpha)\cdots\Gamma(\alpha)}
   {\Gamma(\alpha+y_1)\cdots\Gamma(\alpha+y_p)\over \Gamma(p\alpha +z)}
   {1\over y_1!\cdots y_p!},
\end{array}
\eeq  
with $K(z)=\int_0^\infty \exp(-\gamma)\gamma^z q(\gamma)\,\dd\gamma$
as in Section \ref{section:Bayes}. 
If in particular $\gamma\sim\Gam(\alpha_0,\beta_0)$, then 
\beqn
f(y_1,\ldots,y_p)={\Gamma(\alpha_0+z)\over \Gamma(p\alpha+z)}
   {\Gamma(p\alpha)\over \Gamma(\alpha_0)}
   {\Gamma(\alpha+y_1)\over \Gamma(\alpha)}
   \cdots {\Gamma(\alpha+y_p)\over \Gamma(\alpha)}
   {\beta_0^{\alpha_0}\over (\beta_0+1)^{\alpha_0+z}}
   {1\over y_1!\cdots y_p!}. 
\eeqn 
The point is that this has a multiplicative independence structure 
only if $\alpha_0=p\alpha$, then with negative binomial
marginal distributions. In this light, the (\ref{eq:extra13}) 
construction amounts to an extended class of models 
for count data, allowing both positive and negative correlations.
The generalisation to the nonsymmetric case 
takes $\pi\sim\Dir(\alpha_1,\ldots,\alpha_p)$,
and leads to 
\beqn
f(y_1,\ldots,y_p)
   =K(z){\Gamma(\alpha_1+\cdots+\alpha_p)\over 
   \Gamma(\alpha_1)\cdots\Gamma(\alpha_p)}
   {\Gamma(\alpha_1+y_1)\cdots\Gamma(\alpha_p+y_p)
   \over \Gamma(\alpha_1+\cdots+\alpha_p+z)}{1\over y_1!\cdots y_p!},  
\eeqn 
with the particular choice $\gamma\sim\Gam(\sumip\alpha_i,\beta)$
yielding independent negative binomials. 


\subsection{Poisson processes with dependence in time and space.}

Suppose independent Poisson processes $Y_1(t),\ldots,Y_k(t)$
are observed over a time period $[0,\tau]$, 
and divide this period into $p$ cells or windows, 
say $(t_{j-1},t_j]$. We take $Y_i(t)$ to have cumulative 
intensity function $G_i(t)$, with consequent Poisson parameters 
$\theta_{i,j}=G_i(t_j)-G_i(t_{j-1})$ for the counts 
$Y_{i,j}=Y_i(t_j)-Y_i(t_{j-1})$,  for $j=1,\ldots,p$. 
There are now different ways of modelling the 
$k\times p$ matrix of rate parameters, using aspects
of the apparatus above. For simplicity of presentation
we limit these brief pointers to the neutral cases,
where the vectors of fractions involved come from 
symmetric Dirichlet distributions.  

{\it Idea (a)} is to allow for dependence over time, for each process: 
\beqn
\theta_{i,j}=\gamma_i \pi_{i,j} 
   \quad {\rm for\ }j=1,\ldots,p, 
   {\rm\ with\ }\gamma_i=\sum_{j=1}^p \theta_{i,j},
\eeqn 
and with $(\pi_{i,1},\ldots,\pi_{i,p})$ from a 
$\Dir(\alpha_i,\ldots,\alpha_i)$.
The particular case of $\gamma_i\sim\Gam(p\alpha_i,\beta_i)$
corresponds to independent $\theta_{i,j}\sim\Gam(\alpha_i,\beta_i)$
for $j=1,\ldots,p$.  
{\it Idea (b)} is to build dependence structure 
into the sequence of processes: 
\beqn
\theta_{i,j}=\kappa_j\pi_{i,j} 
   \quad {\rm for\ }i=1,\ldots,k, 
   \quad {\rm\ with\ }\kappa_j=\sum_{i=1}^k \theta_{i,j}, 
\eeqn 
and with $(\pi_{1,j},\ldots,\pi_{k,j})$ from a 
$\Dir(\alpha_j,\ldots,\alpha_j)$. If in particular 
$\kappa_j\sim\Gam(k\alpha_j,\beta_j)$, then we have 
independent $\theta_{i,j}\sim\Gam(\alpha_j,\beta_j)$
for $i=1,\ldots,p$. 

Both ideas (a) and (b) have Bayesian counterparts, and 
motivate extensions of the $L_c$ loss function. 
Suppose we are interested in precise estimates 
of the full $k\times p$ parameter matrix and 
in the cumulatives $\gamma_i = G_i(\tau)$. 
A natural loss function is then 
\beqn
L_c^*(\theta,\delta) = \sum_{i=1}^k L_{c}(\theta_i,\delta_i)
= \sum_{i=1}^k\Big\{\sum_{j=1}^p(\delta_{i,j} - \theta_{i,j})^2/\theta_{i,j} + 
c\big(\sum_{j=1}^p \delta_{i,j} - \gamma_i\big)^2/\gamma_i \Big\},
\eeqn 
where $L_c$ is the loss function introduced in 
Section \ref{section:closs}, and $c$ is a positive constant 
set by the statistician. The Bayes solution is 
\beqn
\delta_{i,j}^B = \frac{1+c}{1+ca_i/b_i }\,a_{i,j}  
   \quad\text{for}\quad j=1,\ldots,p,\;i=1,\ldots,k, 
\eeqn 
where $a_{i,j} = \{\E\,(\theta_{i,j}^{-1}\mid y)\}^{-1}$, 
$a_i = \sum_{j=1}^pa_{i,j}$, and $b_i = \{\E\,(\gamma_i^{-1}\mid y)\}^{-1}$. 
Suppose prior knowledge dictates that intensities of the 
$k$ processes are functionally somewhat alike, 
but at different levels; then a natural prior construction takes 
$(\pi_{i,1},\ldots,\pi_{i,p}) \equiv 
   (\pi_{1},\ldots,\pi_{p}) \sim \Dir(\alpha_1,\ldots,\alpha_p)$ 
(i.e., one draw for all $i$) and 
$\gamma_i \sim{\rm Gamma}(\alpha_{0,i},\beta_{0,i})$ 
for $i=1,\ldots,k$, with these being independent and also 
independent of $(\pi_1,\ldots,\pi_p)$. The posterior is then 
\beqn
\begin{split}
(\pi_1,\ldots,\pi_p)\mid{\rm data} 
   &\sim {\rm Dir}(\alpha_1 + Z_1,\ldots,\alpha_p + Z_p),\\
\gamma_i\mid{\rm data} 
   & \sim \Gam(\alpha_{0,i} + Y_i(\tau),\beta_{0,i} + 1),
   \quad\text{independent for}\quad i = 1,\ldots,k, 
\end{split}
\notag
\eeqn 
where $Z_j = \sum_{i=1}^k \{Y_{i}(t_j) - Y_{i}(t_{j-1})\}$ 
for $j=1,\ldots,p$. The Bayes estimator that emerges is 
\beq
\label{eq:3directions_estimatorA}
\delta_{i,j}^B = \frac{(1+c)(p\bar{\alpha} + Z_j - 1)}{(1+c)(p\bar{\alpha} 
   + Z - 1) - c(p-1)}\frac{a_{0,i} + Y_i(\tau) - 1}{\beta_{0,i} + 1}, 
\eeq 
for $j=1,\ldots,p$ and $i=1,\ldots,k$,
where $\bar\alpha = (1/p)\sum_{j=1}^p \alpha_j$
and $Z$ is the total sum $\sum_{j=1}^p Z_j$.  
This and the accompanying natural frequentist estimator
\beq
\label{eq:3directions_estimatorB}
\hatt\theta_{i,j} 
= \frac{(1+c)(p-1+ Z_j)}{(1+c)(p + Z - 1) - c(p-1)} Y_{i,j} 
= \frac{(1+c)(p-1+ Z_j)}{p-1+(1+c)Z} Y_{i,j} 
\eeq 
are interesting because they borrow information 
in all directions, so to speak; 
cross-sectionally through $Z_j$; 
in time through $Y_i(\tau)$; and 
both horizontally and vertically via the total sum $Z$.
From Section \ref{section:closs} we know that if the two parts of 
$L_c^*$ are viewed separately, the estimators 
\beqn
\delta_{i,j}^{\prime} = \Bigl( 1 - \frac{kp - 1}{kp - 1 + Z}\Bigr)Y_{i,j}
   \quad \text{and}\quad 
\delta_{i}^{\prime}= \Bigl( 1 - \frac{k - 1}{k - 1 + Z}\Bigr)Y_{i}(\tau) 
\eeqn 
uniformly dominate $Y_{i,j}$ and $Y_i(\tau)$, under the loss functions 
$\sum_{i=1}^k\sum_{j=1}^p (\delta_{i,j} - \theta_{i,j})^2/\theta_{i,j}$ 
and $\sum_{i=1}^k(\delta_i - \gamma_i)^2/\gamma_i$, respectively.
The estimators 
(\ref{eq:3directions_estimatorA})--(\ref{eq:3directions_estimatorB})
provide guidance on how to exploit the multivariate nature 
of the problem in order to compromise between $\delta_{i,j}^{\prime}$ 
and $\delta_{i}^{\prime}$, and thereby achieve risk dominance 
in large parts of the parameter space, or even uniformly.

\section{Concluding remarks} 
\label{section:concluding}

We round off our paper by offering a list of 
concluding remarks, some pointing to further research. 

{\it A. Normal approximations and the square-root transformation.}
When the $\theta_i$ are likely to not being small, 
normal approximations might work well, 
and multiparameter estimation may proceed via
e.g.~the approximate model $2y_i^{1/2}\sim\N(2\theta_i^{1/2},1)$. 
The point is that there is a voluminous literature on shrinkage 
methods for normal setups, with Stein--James 
estimators etc. This also invites loss functions 
of the type $\sumip (\delta_i^{1/2}-\theta_i^{1/2})^2$, 
which may also be motivated via the Hellinger distance
between the real and estimated Poisson distributions.
Methods of our paper, using exact Poisson calculations
as opposed to normal approximations, may indeed be used 
to demonstrate that estimators of the type 
$\hatt\theta_i^{1/2}=\{1-g(z)\}y_i^{1/2}$ may be found,
which dominate the default method's $y_i^{1/2}$
in large parts of the parameter space. Specifically,
if $p\ge2$, and given a lower positive threshold $\pi_0$,
methods exist which dominate the default method 
in the region characterised by having all 
$\pi_i=\theta_i/\gamma\ge\pi_0$.

{\it B. Bayesian nonparametrics.}
Suppose a time-inhomogeneous Poisson process 
$Y=\{Y(t)\colon t\ge0\}$ is observed over some
time window $[0,\tau]$, with cumulative intensity
$G=\{G(t)\colon t\ge0\}$. We wish to estimate this
function, with the loss function 
\beqn
L(G,\hatt G)=\int_0^\tau \{\hatt G(t)-G(t)\}^2/G(t)\,w(\dd t),
\eeqn
with some fixed weight measure $w$. Note however
that this is not an obvious extension of our 
previous Clevenson--Zidek loss function, since it works 
via the cumulatives. A natural estimator is $Y$ itself, 
with constant risk function $r_0=\int_0^\tau w(\dd t)$. 
A natural class of priors takes $G\sim\Gam(aG_0,a)$, 
say, a Gamma process with 
independent increments and $G(t)\sim\Gam(aG_0(t),a)$. 
Then $G(t)\midd\data\sim\Gam(aG_0(t)+Y(t),a+1)$, and 
the Bayes estimator becomes 
\beqn
\hatt G(t)={1\over \E\,\{1/G(t)\midd\data\}}
   ={aG_0(t)+Y(t)-1\over a+1}. 
\eeqn 
Working with expressions for the minimum Bayes risk 
one may show that this converges as $a\arr0$
to the value $r_0$, proving that the estimator $Y$
is minimax. 

A larger class of priors can however also be investigated,
inspired by methods and results of our Section \ref{section:Bayes}.
Write $G=\gamma F$, with $\gamma=G(\tau)$ the full mass
and $F=G/\gamma$ normalised to be a cumulative distribution
function. Now construct a prior by having $\gamma$
from some density $q(\gamma)$ independent of 
a Dirichlet process for $F$, say $F\sim\Dir(bF_0)$,
i.e.~prior mean $F_0$ and $b$ the mass or precision parameter.  
An appropriate extension of our Lemma \ref{lemma:lemma1} 
then shows that (i) $\gamma$ and $F$ are independent,
given data; (ii) $\gamma\midd\data$ has a density
proportional to $q(\gamma)\gamma^z\exp(-\gamma)$,
with $z=Y(\tau)$; and (iii) $F\midd\data$ is 
a Dirichlet $bF_0+Y$. The Bayes estimator becomes 
\beqn
G^*(t)=\{\E\,(\gamma^{-1}\midd\data)\}^{-1}
   [\E\,\{F(t)^{-1}\midd\data\}]^{-1}
   ={K(z)\over K(z-1)}{bF_0(t)+Y(t)-1\over b+Y(\tau)-1},  
\eeqn
with $K(z) = \int_0^{\infty} \gamma^z e^{-\gamma}q(\gamma)\,\dd\gamma$
as in Section \ref{section:Bayes}. 
Note that this construction provides a genuine
extension of the Gamma process, and still with 
easy control of the posterior distribution for 
$G=\gamma F$. The usual Gamma process, 
with $G(t)\sim\Gam(aG_0(t),a)$, corresponds to the 
special case of $G=\gamma F$, where 
$\gamma$ has a $\Gam(aG_0(\tau),a)$ and is 
independent of $F\sim\Dir(aG_0)$, i.e.~with 
prior mean $F_0=G_0/G_0(\tau)$ and mass parameter $aG_0(\tau)$. 

{\it C. Separating sum and proportions.} 
When estimating Poisson parameters $\theta_1,\ldots,\allowbreak\theta_p$,
another type of loss function than those worked
with above is to separate the sum $\gamma$ 
and proportions $\pi_i=\theta_i/\gamma$, and 
then work with 
$(\hatt\gamma-\gamma)^2/\gamma
   +z\sumip(\hatt\pi_i-\pi_i)^2/\pi_i$. 
The maximum likelihood procedure corresponds to 
$\hatt\pi_i=Y_i/Z$ and $\hatt\gamma=Z$, 
having constant risk function $1+(p-1)=p$.
Procedures achieving lower risk in large 
parameter regions may be constructed via 
empirical Bayes arguments. 

{\it D. Shrinking towards submodels.}
There is scope for extension of our methods and 
constructions in several directions for multiparameter 
Poisson- and Poisson-related inference. 
It is inherently useful to shrink raw estimates
towards meaningful submodels, such as with 
log-linear setups for analysis of tables 
of count-data, see e.g.~\citet{Agresti2019}.
There are notable Poisson-related models 
for small-area estimation, involving also
mixed Poisson models, zero-inflated versions, etc.
Generally speaking Bayes and empirical Bayes 
constructions can be brought to such tables,
and will tend to work well for sizeable 
parameter regions, whereas the exact risk calculations
worked with in the present article are harder
to generalise.   

\section*{Acknowledgments} 

The authors are grateful for partial funding 
from the PharmaTox Strategic Research Initiative 
at the Faculty of Mathematics and Natural Sciences
(for E.Aa.S's PhD work), and from the Norwegian Research 
Council for the five-year research project group 
FocuStat: Focus Driven Statistical Inference with Complex Data,  
at the Department of Mathematics (led by N.L.H.), 
both at University of Oslo. 


\appendix
\section{Appendix}
\label{app:A}
Here we prove the risk function expression in (\ref{eq:risk.in.risk}),
needed in Theorem \ref{th:admissibility}.  
For the $L_c(\theta,\delta)$ loss function, 
and for estimators of the form 
\beqn
\delta_i = \frac{(1+c)y_i}{p-1+(1+c)z}\,\kappa(z), 
\eeqn
considered in (\ref{eq:symform}), we have 
\beq
\begin{split}
R(\delta,\theta) 
& = \E_{\theta}\Bigl\{ \sumip \frac{1}{\theta_i}(\delta_{i} - \theta_i)^2  
   + \frac{c}{\gamma} \Bigl(\sumip\delta_{i} - \gamma\Bigr)^2\Bigr\}  \\
& = \E_{\theta} \sumip \Bigl\{ \frac{(1+c)^2Y_i^2/\theta_i}
   {(p-1+(1+c)Z)^2}\kappa^2(Z)  
   - 2 \frac{(1+c)Y_i}{p-1+(1+c)Z}\kappa(Z) + \theta_i \Bigr\} \\
& \qquad + c\,\E_{\gamma}\Bigl[  \frac{(1+c)^2Z^2/\gamma}
   {\{p-1+(1+c)Z\}^2}\kappa^2(Z) 
   - 2 \frac{(1+c)Z}{p-1+(1+c)Z}\kappa(Z) + \gamma \Bigr] \\
& = \E_{\gamma}\left[ \frac{(1+c)^2}{\{p-1+(1+c)Z\}^2} 
   \frac{Z\{p-1+(1+c)Z\}}{\gamma}\kappa^2(Z) \right. \\
& \qquad\qquad\left. - 2\frac{(1+c)^2Z}{p-1+(1+c)Z}\kappa(Z) + (1+c)\gamma \right] \\
& = \E_{\gamma}\Bigl\{ \frac{(1+c)^2Z}{p-1+(1+c)Z} {\gamma}\kappa^2(Z)
- 2\frac{(1+c)^2Z}{p-1+(1+c)Z}\kappa(Z) + (1+c)\gamma \Bigr\}\\
& = \E_{\gamma} \Bigl[ \frac{(1+c)^2Z }{p-1+(1+c)Z}
   \frac{\{\kappa(Z) - \gamma\}^2}{\gamma} - \frac{(1+c)Z}{p-1+(1+c)Z} \gamma 
   + (1+c)\gamma \Bigr] \\ 
& = \E_{\gamma} \Bigl[ \frac{(1+c)^2Z }{p-1+(1+c)Z}
   \frac{\{\kappa(Z) - \gamma\}^2}{\gamma} 
   + \frac{(p-1)(1+c)\gamma}{p-1+(1+c)Z} \Bigr],
\end{split}
\notag
\eeq
proving what was required. 


\bibliographystyle{biometrika}
\bibliography{poisson_things3.bib} 


\end{document}